\documentclass[a4paper,11pt,reqno]{amsart}

\RequirePackage[l2tabu, orthodox]{nag}
\usepackage[T1]{fontenc}
\usepackage[utf8]{inputenc}
\usepackage{lmodern}
\usepackage[frenchb,english]{babel}
\usepackage{color}
\usepackage{amsthm,amssymb,amsmath}
\usepackage{graphicx}
\usepackage[all]{xy}
\usepackage{enumerate}
\usepackage{pst-node}
\usepackage{tikz-cd}

\def\comment#1{}

\newcommand{\N}{\mathbb{N}}

\newcommand{\D}{\mathbb{D}}

\newcommand{\Ss}{\mathbf{S}}
\newcommand{\T}{\mathbb{T}}
\newcommand{\cE}{\mathcal{E}}
\newcommand{\cC}{\mathcal{C}}
\newcommand{\cS}{\mathcal{S}}

\newcommand{\cA}{\mathcal{A}}

\newcommand{\cL}{\mathcal{L}}

\newcommand{\cT}{\mathcal{T}}

\newcommand{\x}{\mathbf{x}}

\newcommand{\y}{\mathbf{y}}
\newcommand{\z}{\mathbf{z}}

\def\to{\mathop{\rightarrow}}
\def\dans{\mathop{\subset}}
\newcommand{\moins}{\setminus}

\newcommand{\ri}{r_{\mathrm{inj}}}

\newcommand{\sys}{\mathrm{sys}}

\newcommand{\dist}{\mathrm{dist}}
\newcommand{\Int}{\mathrm{Int}}
\newcommand{\Id}{\mathrm{Id}}

\newcommand{\vide}{\emptyset}
\def\dans{\mathop{\subset}}
\def\To{\mathop{\longrightarrow}}

\newtheorem{theorem}{Theorem}[section]
\newtheorem{maintheorem}{Theorem}

\newtheorem{lemma}[theorem]{Lemma}
\newtheorem{proposition}[theorem]{Proposition}

\newtheorem*{theoremc}{Theorem C}

\theoremstyle{definition}
\newtheorem{definition}[theorem]{Definition}

\theoremstyle{remark}
\newtheorem{remark}[theorem]{Remark}

\usepackage[margin=3.5cm]{geometry}

\usepackage[colorlinks=true,linkcolor=blue,citecolor=magenta]{hyperref}

\date{}
\author[S. Alvarez]{S{\'e}bastien Alvarez}
\address{CMAT, Facultad de Ciencias, Universidad de la Rep\'ublica, Uruguay}
\email{salvarez@cmat.edu.uy}

\author[J. Brum]{Joaqu\'in Brum}
\address{IMERL, Facultad de Ingenier\'ia, Universidad de la Rep\'ublica, Uruguay}
\email{jbrum@fing.edu.uy}

\author[M.Martinez]{Matilde Mart\'inez}
\address{IMERL, Facultad de Ingenier\'ia, Universidad de la Rep\'ublica, Uruguay}
\email{mmartinez@fing.edu.uy}

\author[R. Potrie]{Rafael Potrie}
\address{CMAT, Facultad de Ciencias, Universidad de la Rep\'ublica, Uruguay}
\urladdr{www.cmat.edu.uy/$\sim$rpotrie}
\email{rpotrie@cmat.edu.uy}
 
\author[]{an Appendix written with Maxime Wolff}
\address{IMJ, Universit\'e Sorbonne Paris Cit\'e, France}
\email{maxime.wolff@imj-prg.fr}

\title{Topology of leaves for minimal laminations by hyperbolic surfaces}
\thanks{The authors were partially supported by CSIC 618, CSIC I+D 389, FCE-135352, FCE-148740 and MathAmSud RGSD 19-MATH-04. J.B. acknowledges the support of CONICYT via FONDECYT Postdoctorate 3190719. S.A. and J.B. were partially supported by the the Distinguished Professor Fellowship of FSMP. S.A. and M.W. acknowledge the support of LIA-IFUM}

\begin{document}

\maketitle

\begin{abstract}
We construct minimal laminations by hyperbolic surfaces whose generic leaf is a disk and contain any prescribed family of surfaces and with a precise control of the topologies of the surfaces that appear. The laminations are constructed via towers of finite coverings of surfaces for which we need to develop a relative version of residual finiteness which may be of independent interest. The main step in establishing this relative version of residual finiteness is to obtain finite covers with control on the \emph{second systole} of the surface, which is done in the appendix. In a companion paper, the case of other generic leaves is treated. 
\end{abstract}

\section{Introduction} 
A \emph{lamination} or {\em foliated space} of dimension $d$ is a compact metrizable space which is locally homeomorphic to a disk $D$ in $\mathbb{R}^d$ times a compact set $T$ called \emph{transversal}. The space is required to have a compatibility condition between these local trivialisations, which guarantees that sets of the form $D\times\{t\}$ glue together to form $d$-dimensional manifolds called {\em leaves}. The space is therefore a disjoint union of the leaves, which can be embedded in the compact space in very complicated ways. When the space is a manifold, this structure is usually called a {\em foliation}. When the transversal $T$ is a Cantor set, it is called a {\em solenoid}. When all the leaves are dense, we say that the lamination is {\em minimal}. We refer the reader 
to \cite{candel-conlon} for more details and to \cite{Ghys_Laminations} for an excellent survey about the two-dimensional case.

For laminations by surfaces, i.e. when $d=2$, the topology and the geometry of leaves have been widely studied. Cantwell and Conlon proved that any surface appears as a leaf of a foliation by surfaces of a closed 3-manifold, see \cite{CC_realization}. Note that their construction does not produce a minimal foliation. In a very nice work \cite{MG}, Gusm{\~a}o and Meni{\~n}o have recently shown how to construct \emph{minimal} foliations by surfaces on some finite quotients of circle bundles over closed surfaces containing any prescribed countable familly of noncompact surfaces as leaves. On the other hand, Ghys proved in \cite{Ghys_generic} that for a lamination by surfaces, the generic leaf --in the sense of Garnett's harmonic measures \cite{Garnett}-- is either compact or in a list containing only six noncompact surfaces. An analogous statement holds for generic leaves in a topological sense under an assumption which is satisfied by minimal laminations \cite{Cantwell_Conlon}. Much work on the quasi-isometry class of 
leaves has led to many results concerning the topology of generic leaves. This is the subject of \cite{alvarez-candel}, which contains most relevant results and references. Many 
remarkable examples have given us insights into which leaves can appear, or coexist, in a lamination by surfaces. It is worth mentioning the Ghys-Kenyon example, constructed in \cite{Ghys_Laminations} which is a \emph{minimal} lamination containing leaves with different conformal types. Also one has the famous Hirsch's foliation where all leaves have infinite topological types (see \cite{Hirsch1975} for the original construction and \cite{ADMV,AlvarezLessa,Cantwell_Conlon,Ghys_generic} for the minimal construction). Other aspects of the study of this subject have been pursued in different works, a non exhaustive list is \cite{Clark_Hurder,Hurder_matchbox,Sibony_etc,McCord,Schori,Sullivan,Verjovsky}. 

In this paper, we are interested in the study of the topology of the leaves for minimal laminations by {\em hyperbolic} surfaces. Such laminations are quite ubiquituous after a very beautiful uniformisation result of Candel \cite{Candel} which shows that unless some natural obstruction appears, every lamination by surfaces admits a leafwise metric which makes every leaf of constant negative curvature. 

The motivation for this work was to understand the possible topologies that can coexist in a minimal lamination by hyperbolic surfaces. In this setting, a strong dichotomy holds: either the generic leaf is simply connected or all leaves have a `big' fundamental group --one which is not finitely generated-- see \cite[Theorem 2]{ADMV}. We have found that there are no obstructions to having all surfaces simultaneously in the same lamination.

\begin{maintheorem}\label{t.uno}
There exists a minimal lamination by hyperbolic surfaces so that for every non-compact open surface $S$ there is at least one leaf homeomorphic to $S$.
\end{maintheorem}

After proving this theorem, we were informed that a similar result had been announced in the late 90's by Blanc \cite{Blanc_these}, as part of his unpublished doctoral thesis. Nevertheless, the lamination described by Blanc, constructed with a refinement of Ghys-Kenyon's method, does not seem to admit a leafwise metric of constant curvature $-1$.

We present a very  flexible combinatorial method to construct minimal laminations by hyperbolic surfaces with prescribed surfaces as leaves. This yields, in a straightforward way, the example announced in Theorem~\ref{t.uno} and many others. For minimal laminations, there is always a residual set of leaves having 1, 2 or a Cantor set of ends, see \cite{Cantwell_Conlon}. In this paper we restrict ourselves to considering laminations for which there is a residual set of leaves which are planes. Blanc studied the two-end case in \cite{Blanc_2bouts}. 

A companion paper \cite{AB} by the first two authors will address the case where the generic leaf has a Cantor set of ends. It contains a refinment of \cite[Theorem 2]{ADMV}: \emph{for every leaf of such a lamination, all isolated ends are accumulated by genus ---this is called condition $(\ast)$ in} \cite{AB}. Using the formalism developped in the present paper, as well as new techniques, it is proven that this is the only obstruction. Better still: \emph{there exists a minimal lamination $\cL$ by hyperbolic surfaces whose generic leaf is a Cantor tree and such that for every non-compact surface $S$ satisfying condition $(\ast)$, there is at least one leaf homeomorphic to $S$.} The formalism used there and the spirit of the proof are very similar, but other difficulties appear, and new techniques are needed.


We are unable, in general, to prescribe exactly which noncompact surfaces will appear as leaves in our examples. However, we do not know if this is a weakness of our method or if it reflects a general obstruction. When restricting ourselves to finite or countable families of surfaces (recall that there are uncountably many topological types of them), the formalism of \emph{forests of surfaces} together with Theorem \ref{teo-mainabstract}, allows us to get an optimal result.

\begin{maintheorem}\label{t.tres}
 Let $\cA= \{S_n\}_{n}$ be a  finite or countable sequence of non-compact open surfaces different from the plane. 
Then,
there is a minimal lamination by hyperbolic surfaces for which the generic leaf is a plane, and the leaves which are not simply connected form a sequence $\{L_n\}$ such that $L_n$ is homeomorphic to $S_n$ for every $n$.
\end{maintheorem}

 Notice that the sequence $\cA$ can take any value more than once --even infinitely many times. So this theorem says that, given a finite or countable set of noncompact surfaces, there is a lamination having each element as a leaf exactly some prescribed number of times. 
 
These techniques can produce a wider variety of examples. We refer the reader to Theorem~\ref{teo-mainabstract} and Proposition~\ref{p.forestinclusion} for the most general statements which allow in particular to show Theorems \ref{t.uno} and \ref{t.tres} (see section \ref{s.including_forest}).

\begin{remark}\label{rem.obstsusp}
Foliations of codimension one do not have this flexibility, at least in enough regularity. In fact, for a foliation by surfaces of a closed 3-manifold by surfaces of finite type, either all leaves are simply connected or there are infinitely many leaves which are not. This will be explained in Proposition~\ref{rem.obstminimalfol} (see also Remark~\ref{rem.obssuspS1}).
\end{remark}

All of our examples are solenoids, obtained as the inverse limit $\mathcal{L}$ of an infinite tower 
$$\cdots \xrightarrow{p_n} \Sigma_n \xrightarrow{p_{n-1}} \Sigma_{n-1} \cdots \xrightarrow{p_0}\Sigma_0$$
of finite covers of a compact hyperbolic surface $\Sigma_0$. 
These solenoids are also the object of \cite[Section 2]{Sibony_etc}. There, Sibony, Fornaess and Wold prove, among other things, that there is a unique transverse holonomy-invariant measure, and also that there are no harmonic measures other than the one which is totally invariant (see \cite[Theorem 1]{Sibony_etc}, bearing in mind that holonomy-invariant measures are the same as positive closed currents, as explained in \cite{Sullivan_currents}).
Also, they prove  that the laminations we construct embed in $\mathbb{C}\mathbb{P}^3$ (see also \cite{Deroin_plongements} for other results of immersion of laminations inside projective complex spaces).
On the other hand, it is easy to see that laminations obtained by inverse limits hardly ever embed in a 3-manifold. This follows from the more general fact that a minimal lamination by hyperbolic surfaces which admits a holonomy-invariant measure and embeds in a 3-manifold has the following property: either all its leaves are simply connected or none of them are, see Remark \ref{remark-embed}. 

All covering maps $p_n:\Sigma_{n-1}\to \Sigma_n$ are local isometries, and an  appropriate control of the geometry of the $\Sigma_n$ will allow us to prescribe the topology of the leaves of $\mathcal{L}$. Fornaess, Sibony and Wold construct a lamination where all leaves but one are simply connected. In the present work, to be able to construct every possible surface, we need a tighter grip on the properties of the tower, and therefore a better understanding of finite coverings of surfaces. The main technical tool is the following statement, of independent interest, concerning covering maps between compact hyperbolic surfaces (which appears in the Appendix, joint with M. Wolff).

\begin{maintheorem}
\label{teo.apendice}
  Let $\Sigma$ be a closed hyperbolic surface, and let $\alpha\subset\Sigma$
  be a simple closed geodesic. Then, for all $K>0$, there exists a finite 
  covering $\pi:\hat \Sigma\to\Sigma$ such that
  \begin{itemize}
  \item $\hat \Sigma$ contains a non-separating simple closed geodesic such that $\pi(\hat \alpha)=\alpha$ and $\pi$ restricts to a homeomorphism on $\hat \alpha$;
  \item every simple closed geodesic which is not $\hat \alpha$ has length larger than $K$.
  \end{itemize}
\end{maintheorem}

In other words, there is a finite cover where the curve $\alpha$ has a $(1:1)$-lift, all the other curves open up and no new short curve appears. This result allows us to get a relative version of residual finiteness for surface groups which may be interesting by itself, see Theorem \ref{t.finitecovering}.

It is similar in spirit to the LERF property proved for surface groups by Scott in \cite{Scott}. By directly applying the LERF property we could find a finite covering where the curve $\alpha$ has a $(1:1)$-lift with a large collar neighbourhood but some new short curves can appear in its complement.  Similar geometric and quantitative properties of surface groups have been proved with different motivations, for a recent such result see, e.g. \cite{lazarovich-levit-minsky}.


\paragraph{ {\bf Organization of the paper --}} The paper is structured as follows: Section~\ref{s.Preliminaries} covers preliminary material related to compact hyperbolic surfaces, towers of coverings of such surfaces and properties of their inverse limits. In Section~\ref{s.illustration} we present some illustrative examples, motivating the techniques and pointing out some differences with the foliation setting, it closes with some explainations on how the general results will be obtained.  Section~\ref{s.toolbox} develops the necessary tools to control the geometry of the finite covers (in particular, the general version: Theorem \ref{t.finitecovering} of the relative version of residual finitness is obtained). In Section~\ref{s.admissible} we define an abstract object, an \emph{admissible tower of coverings}, which enables us to control the topology of leaves of a lamination. The notion of forest of surfaces is also introduced. Section \ref{s.including_forest} is the technical heart of the paper: we 
prove there Proposition \ref{p.forestinclusion}, which allows us to construct towers of finite coverings admissible with respect to any forest of surfaces. Finally, in Section \ref{s.final_proof} we construct the necessary forests of surfaces in order to prove Theorems \ref{t.uno} and \ref{t.tres}; the constructions there are flexible and allow to make other examples that the reader can pursue if desired.

{\small
\paragraph{ {\bf Acknowledgements --}} It is a pleasure to thank Henry Wilton whose answer to our question in MathOverflow, which contained a first sketch of proof of Theorem \ref{teo.apendice} (see \cite{HW}), has been very important for the completion of our work. Gilbert Hector kindly communicated to us Blanc's thesis, we are thankful to him. Finally we thank Fernando Alcalde, Pablo Lessa, Jes\'us \'Alvarez Lopez, Paulo Gusm\~ao and Carlos Meni\~no for useful discussions. Last but not least we wish to thank the referee for his/her valuable comments that allowed us to improve the presentation of this work.
}

\section{Preliminaries}\label{s.Preliminaries}

\subsection{Towers of coverings and minimal laminations}

\paragraph{ {\bf Definitions --}} Let $\T = \{p_n : \Sigma_{n+1} \to \Sigma_n \}$ where $\Sigma_n$ are closed hyperbolic surfaces and $p_n$ are finite (isometric) coverings. We define $\mathcal{L}$ to be the inverse limit of $\T$, that consists of sequences $\x=(x_n)_{n\in\N}\in\prod_n \Sigma_n$ such that for every $n\in\N$, $p_n(x_{n+1})=x_n$, and we endow it with the the topology induced by the product topology.

\begin{remark}
Let us emphasize that in this paper all covering maps are local isometries.
\end{remark}

The set $\mathcal{L}$ is a compact space and possesses a lamination structure so that the leaf of a sequence $\x=(x_n)_{n\in\N}$, denoted by $L_\x$, is formed by those sequences $\y=(y_n)_{n\in\N}$ such that $\dist(x_n,y_n)$ is bounded (see \cite[Proposition 2]{Sibony_etc}). 

\begin{remark}\label{r.distanciamonotona} Let $\x=(x_n)_{n\in\N}$ and $\y=(y_n)_{n\in\N}$ be two different points in $\mathcal{L}$. Then the sequence of distances $d(x_n,y_n)$ is increasing with $n$. To see this notice that a path $\alpha$ between $x_n$ and $y_n$ that realizes the distance between them, projects down onto a path of the same length between $x_m$ and $y_m$ whenever $m\leq n$.
 \end{remark}

\paragraph{{\bf Leafwise metric --}}Let $\x\in\cL$ and $L_\x$ be the leaf of $\x$. Let us consider the following covering maps
\begin{itemize}
\item $\Pi_n:L_\x\to \Sigma_n$ associating to $\y$ the $n$-th coordinate $y_n$.
\item $P_n=p_0\circ\ldots \circ p_{n-1}: \Sigma_n\to \Sigma_0$.
\item $P_{n,m}=p_m\circ\ldots\circ p_{n-1} :\Sigma_n\to\Sigma_m$ for $m<n$.
\end{itemize}

Note that $P_n \circ \Pi_n =\Pi_0$ for every $n$ and that $p_{n-1} \circ \Pi_n=\Pi_{n-1}$. We can lift the metric of $\Sigma_0$ on each $\Sigma_n$ using maps $P_n$ (so all coverings $p_n$ are local isometries) and on $L_\x$ (so that all $\Pi_n$ are local isometries). We denote by $g_n$ the metric on $\Sigma_n$ and by $g_{L_\x}$ the metric on $L_\x$. This gives a \emph{leafwise metric}, i.e. an assignment $L\mapsto g_L$ which is transversally continuous in local charts.

\medskip

\paragraph{{\bf Minimality --}} Recall that a lamination is said to be \emph{minimal} if all of its leaves are dense.

\begin{proposition}\label{p.minimality}
The lamination $\cL$ defined by a tower $\T$ of finite coverings of closed hyperbolic surfaces is minimal.
\end{proposition}

\begin{proof}
Given points $\x=(x_n)_{n\in\mathbb{N}}$ and $\y=(y_n)_{n\in\mathbb{N}}$ in $\mathcal{L}$, we show that for any neighbourhood $\hat U$ of $\x$, $\hat U\cap L_{\y}\neq \emptyset$.

Recall that the topology on $\cL$ is induced by the product topology. So given an open neighbourhood $\hat U$ of $\x$ in $\cL$ there exists an integer $n_0>0$ and a positive number $\delta_0$, such that $\hat U$ contains every point $\x'\in\cL$ satisfying $\dist(x_n,x'_n)<\delta_0$ for every $n\leq n_0$.

Let $\alpha_{n_0}$ be any path in $\Sigma_{n_0}$ starting at $y_{n_0}$ and ending at $x_{n_0}$.
 For every $n\geq n_0$ there exists a path $\alpha_n$ in $\Sigma_n$ starting at $y_n$ such
  that $P_{n,n_0}\circ\alpha_n=\alpha_{n_0}$. Note that for all $n\geq n_0$ we have that the length $l_{\alpha_n}$ of $\alpha_n$ is equal to the length 
  $l_{\alpha_{n_0}}$ of $\alpha_{n_0}$. Let $\y'\in\cL$ be defined as follows. For $n\leq n_0$, $y_n'=x_n$ and for
   $n>n_0$, $y_n'$ is the other extremity of $\alpha_n$. The first condition implies that $\y'\in\hat U
   $. The second one implies that $\dist(y_n,y_n')\leq l_{\alpha_{n_0}}$ for every $n$ so that $\y'\in L_\y$.
    This proves that $\hat U\cap L_\y\neq\vide$. 
    
    This proves the minimality of $\cL$.
\end{proof}

In fact, we know from \cite{Matsumoto} that laminations constructed in this way must be \emph{uniquely ergodic}, since it can be shown that they are equicontinuous (see also \cite{Sibony_etc}). 

\begin{remark}\label{remark-embed}
Laminations constructed this way that embed in 3-manifolds need to be quite special. Indeed, since they admit a transverse invariant measure, the codimension one property implies some local order preservation: If $\Lambda$ is compact lamination with a transverse invariant measure,  $i: \Lambda \to M$ is an embedding in a 3-manifold and if $L$  is a non-simply connected leaf, then, one can consider a small transversal to a non trivial loop $\gamma$ and the holonomy of the lamination can be pushed to nearby leaves because the measure is preserved as well as the order. The loops in the nearby leaves cannot be homotopically trivial since that would imply that closed curves of a given length bound arbitrarily large disks contradicting the fact that leaves are hyperbolic. If the lamination is minimal this implies that every leaf has a non-trivial fundamental group. This implies that if a minimal lamination with hyperbolic leaves and a transverse invariant measure embeds in a 3-manifold then either every leaf is 
simply connected, or no leaf is.
\end{remark}

\subsection{Geometry and topology of the leaves}
\label{s.geom_topo_leaves}

\paragraph{ {\bf Cheeger-Gromov convergence --}} 

A sequence $(\Sigma_n,g_n,x_n)_{n\in\N}$ of pointed complete Riemannian manifolds is said to converge towards the pointed complete Riemannian manifold $(L,g,x)$ in the \emph{Cheeger-Gromov sense} whenever there exists a sequence of smooth mappings $\Pi_n:L\to\Sigma_n$ such that
\begin{enumerate}
\item for every $n\in\N$, $\Pi_n(x)=x_n$; and for every compact set $K\dans L$ there exists an integer $n_0=n_0(K)>0$ such that
\item for every $n\geq n_0$, $\Pi_n$ restricts to a diffeomorphism of $K$ onto its image;
\item the sequence of pull-back metrics $(\Pi_n^{\ast}g_n)_{n\geq n_0}$ converges to $g$ in the $C^{\infty}$-topology over $K$.
\end{enumerate}

The sequence $(\Pi_n)_{n\in\N}$ is called a sequence of \emph{convergence mappings} of $(\Sigma_n,g_n,x_n)_{n\in\N}$ with respect to $(L,g,x)$. This mode of convergence is sometimes called \emph{smooth convergence}: \cite{Pablito,Pet}. It  appeared first in \cite{Gr}, where Gromov proved that Cheeger's finiteness theorem (see \cite{Ch}) was in fact a compactness result. We will refer to \cite{Pet} for more details about it.

\paragraph{ {\bf Topology of the leaves --}} Cheeger-Gromov convergence proves to be especially useful to identify the topology of leaves of a lamination coming from a tower of finite coverings.

Below, $\cL$ denotes the inverse limit of a tower $\T=\{p_{n+1}:\Sigma_{n+1}\to\Sigma_n\}$ of finite coverings of closed hyperbolic surfaces.
\begin{proposition}\label{p.Cheeger_Gromov}
Let $\x=(x_n)_{n\in\N}\in \mathcal{L}$. Then the pointed leaf $(L_\x,g_{L_\x},\x)$ is the Cheeger-Gromov limit of pointed Riemannian manifolds $(\Sigma_n,g_n,x_n)$.
\end{proposition}

\begin{proof}
Candidates for convergence mappings are given by the maps $\Pi_n:L_\x\to\Sigma_n$. These are indeed local isometries. By the inverse function theorem it is enough to prove that for every $R>0$ there exists $n_0$ such that $\Pi_n$ is injective on the ball $B_{L_\x}(\x,R)$ for every $n\geq n_0$.

Consider the groups $G_n=(P_n)_\ast(\pi_1(\Sigma_n,x_n))$ and $G_\x=(\Pi_0)_\ast(\pi_1(L_\x,\x))$. By definition they form a decreasing sequence of subgroups of $\pi_1(\Sigma_0,x_0)$. On the other hand, for every $R>0$ there are only finitely many geodesic loops of length less than $R$ in $\Sigma_0$. Thus, for every $R>0$ there exists $n_1(R)\geq 0$ such that for every $n\geq n_1(R)$ we have
\begin{equation}
\label{eq.group_stat}
G_\x\cap D_R=G_n\cap D_R
\end{equation}
where $D_R$ denotes the disk centered at the identity of radius $R$ inside $\pi_1(\Sigma_0,x_0)$ for the geometric norm, i.e. the one that associates to $\gamma$ the length of the corresponding geodesic loop based at $x_0$ (i.e. the one that associates to $\gamma$ the length of the associated geodesic loop based at $x_0$ that is the projection of the geodesic segment $[x,\gamma.x]$ where $x$ is the preferred lift of $x_0$).

First note that if $\Pi_n(\y)=\Pi_n(\z)$ for some $n\in\N$ and $\y,\z\in L_\x$ then we have $\Pi_m(\y)=\Pi_m(\z)$ for every $m\leq n$. Assume that for infinitely many integers $n$ there exists an open geodesic segment $\alpha_n\dans B(\x,R)$ so that $\Pi_n\circ\alpha_n$ is a geodesic loop (where $B(\x,R)$ denotes the ball of radius $R$ about $\x$). Using Ascoli's theorem and the remark above, we see that there exists an open geodesic ray $\alpha\dans B(\x,R)$ such that $\Pi_n\circ\alpha$ is a closed geodesic loop (hence nontrivial in homotopy) for every $n\in\N$, contradicting \eqref{eq.group_stat}.

\end{proof}

Notice in particular that the leaves of $\cL$ are hyperbolic surfaces, but all this discussion works equally well if one considers towers of coverings of compact Riemannian manifolds of any dimension.


\subsection{Some hyperbolic geometry}\label{s.hyp_geom}

\paragraph{ {\bf Systoles, collars and injectivity radius --}} Below we set definitions and notations of hyperbolic geometry that will be used throughout the paper.

\begin{definition}\label{d.collars1}Let $X$ be a compact hyperbolic surface with geodesic boundary. The \emph{systole} $\sys(X)$ of $X$ is the length of the shortest geodesic in $X$. The \emph{internal systole} of $X$ is the shortest length of an essential and primitive closed curve in $X$ which is not isotopic to a boundary component. This is also the smallest length of a closed geodesic included in the interior $\Int(X)$.
\end{definition}

Notice that if $X$ has no boundary, these two concepts coincide, but when $X$ has boundary, the systole could be achieved by a boundary component. 

\begin{definition}\label{d.collars}
The \emph{(maximal) half-collar width} $K_0$ at a boundary component $\alpha$ of $X$ is the minimal half-distance of two lifts of $\alpha$ to the Poincar\'e disk $\D$. It satisfies that for  every $K<K_0$ the $K$-neighbourhood of $\alpha$ is an embedded half-collar. 

We say that the boundary of $X$ has a \emph{half-collar of width} $K_0$ if there exists a neighbourhood of $\partial X$ consisting of a disjoint union of embedded half-collars of width $K_0$. 
\end{definition}

We now give a series of lemmas that we will use later in the text.

\begin{lemma}
\label{l.halfcollars}
Let $X$ be a  hyperbolic surface with geodesic boundary which is not a pair of pants and $\alpha\dans\partial X$ be a boundary component. Then
$$K_0>\frac{\sigma-l_{\alpha}}{2},$$
where $\sigma$ and $K_0$ denote respectively the internal systole and the half-collar width at $\alpha$ of $X$.
\end{lemma}

\begin{proof}
By definition $2K_0$ is the minimal distance between two lifts of $\alpha$ to the upper half
 plane and it is the length of a geodesic segment $\gamma$ cutting $\alpha$ orthogonally at two points $x$ and $y$, included inside the pair of pants $P$ attached to $\alpha$.

The pair of pants $P$ has a boundary component $\beta$ disjoint from the boundary $\partial X$. This simple closed curve is isotopic to the concatenation of $\gamma$ with a geodesic segment $[x,y]$ included in $\alpha$. We find
$$2K_0+l_\alpha>l_\beta\geq\sigma.$$

The lemma follows.
\end{proof}

The injectivity radius at a point $x$ of a Riemannian manifold will be denoted by $\ri(x)$. In other words, the injectivity radius at $x$ is the smallest length of a geodesic loop based at $x$. Notice that the geodesic loop is not necessarily a closed geodesic since it can have a cone point at $x$.

\begin{lemma}\label{l.inj_radius}
Let $X$ be a hyperbolic surface with geodesic boundary and $\sigma$ be its internal systole. Assume that boundary components of $X$ have disjoint collars of width $K_0>0$. Assume furthermore that we have
$$K_0\leq \sys(X)\cosh\left(\frac{K_0}{2}\right).$$  

Let $x\in X$ such that $\dist(x,\partial X)\geq K_0$. Then
$$\ri(x)\geq\min\left(\sigma,\frac{K_0}{2}\right).$$
\end{lemma}

\begin{proof}
Assume that the hypotheses of the lemma hold. Let $x$ be a point such that $\dist(x,\partial X)\geq K_0$. We must prove that a primitive geodesic loop $\gamma$ based at $x$ satisfies  $l_\gamma\geq\min(\sigma,K_0/2)$. If $\gamma$ is not isotopic to a boundary component of $X$ then $l_\gamma\geq\sigma$. So assume that $\gamma$ is isotopic to a boundary component $\beta$ of $X$ and that $\gamma$ satisfies $l_\gamma<K_0/2$. Then, as a consequence of the triangle inequality, $\gamma$ is entirely contained outside the $(K_0/2)$-neighbourhood of $\beta$. 

Let $\tilde{\beta}$ be a lift of $\beta$ to the Poincar\'e disk $\D$, it is invariant by a hyperbolic isometry denoted by $h$ whose translation length is $l_\beta\geq\sys(X)$. There exists a geodesic segment $\tilde \gamma$ between $\tilde x$ (a lift of $x$) and $h(\tilde x)$ which projects down isometrically onto $\gamma$ and is located outside a $(K_0/2)$-neighbourhood of $\tilde \beta$. Since the orthogonal projection outside a $(K_0/2)$-neighbourhod of $\tilde \beta$ is a contraction of factor $1/\cosh(K_0/2)$ we must have
$$\sys(X)\leq l_{\beta}\leq \frac{l_{\tilde{\gamma}}}{\cosh(K_0/2)}\leq\frac{K_0}{2\cosh(K_0/2)},$$
which contradicts the hypothesis. We deduce that if $\gamma$ is isotopic to a boundary component it must satisfy $l_\gamma\geq K_0/2$.

\end{proof}

\begin{lemma}
\label{l.crosscollars}
Consider $\Sigma$ and $X$ hyperbolic surfaces with geodesic boundary, and a map $$\varphi:X\to\Sigma$$ which is an isometric embedding in restriction to $\Int(X)$. Take $\alpha\dans\partial X$, a boundary component and denote by  $K_0$ the half-collar width of $\alpha$. Then if $\gamma$ is a closed geodesic in $\Sigma$ that crosses $\varphi(\alpha)$, we have $l_{\gamma}>K_0$ 
\end{lemma}

\begin{proof}
Let us denote by $C$ the image by $\varphi$ of a half-collar at $\alpha$ with width $K_0$, as stated in the lemma. Let $\gamma_0$ be a connected component of $\gamma\cap C$ which meets $\varphi(\alpha)$. Since two geodesic arcs cannot bound a bigon, $\gamma_0\cap\varphi(\alpha)$ must be a singleton, so $\gamma_0$ must connect the two boundary components of $C$, therefore it must have length greater than $K_0$. 
\end{proof}

We will also need the following:

\begin{lemma} 
\label{l.systole-gluing}
Let $X$ be a hyperbolic surface with geodesic boundary written as a union $$X=\bigcup_{i=0}^{n}X_i$$ where the $X_i$ are subsurfaces with geodesic boundary meeting each other at boundary components, and $K$ a positive number. Assume moreover that for $i=0,\ldots,n$ we have: 
\begin{itemize} 
\item the internal systole of $X_i$ is greater than $K$;
\item the half collar width of every boundary component of $X_i$ included in the interior of $X$ is greater than $K$;
\item the boundary components of $X_i$ included in the interior of $X$ have length greater than $K$.
\end{itemize} 
Then, the internal systole of $X$ is greater than $K$.

\end{lemma}

\begin{proof} Take a closed geodesic $\gamma\subset \textrm{Int}(X)$, we must check that its length is greater than $K$. For this, we distinguish three cases.

\noindent \emph{Case 1. $\gamma\subset \Int(X_i)$ \textrm{for some} $i$.} In this case, the length of $\gamma$ must be greater or equal than the internal systole of $X_i$, and therefore is greater than $K$ by hypothesis.

\noindent \emph{Case 2. $\gamma$ crosses a boundary component $b$ of $X_i$ for some $i$.} By hypothesis, the half-collar width of every boundary component of $X_i$ included in the interior of $X$ is greater than $K$, then Lemma \ref{l.crosscollars} implies that $l_{\gamma}>K$.

\noindent \emph{Case 3. $\gamma$ is a boundary component of some $X_i$ included in $\Int(X)$.} In this case the length of $\gamma$ is greater than $K$ by hypothesis. 
This finishes the proof of the lemma. 
\end{proof}

\paragraph{ {\bf Retraction on subsurfaces --}} We will need the following proposition to identify the topology of some complete hyperbolic surface knowing that of a subsurface.

\begin{proposition}\label{p.recognize_surface}
Let $L$ be a complete hyperbolic surface without cusps and $S\dans L$ a closed subsurface with geodesic boundary such that every connected component $C_i$ of $L\moins S$ satisfies the following properties.
\begin{enumerate}
\item $C_i$ does not contain a closed geodesic.
\item The boundary $\partial \overline{C}_i$ is connected
\end{enumerate}
Then $L$ is diffeomorphic to $\Int(S)$.
\end{proposition}

\begin{proof}
Let $C_i$ be a connected component of $L\moins\overline{S}$ and $\widetilde{C}_i$ be a component of its preimage to the Poincar\'e disk $\D$ by uniformization. Its closure is geodesically convex and has geodesic boundary (argue like in the proof of \cite[Lemma 4.1.]{Casson_Bleiler}).

Moreover, we can prove that its boundary is connected so this is a half plane. In order to see this we use that the boundary $\partial\overline{C}_i$ is connected so if the closure of $\widetilde{C}_i$ had various boundary components, there would exist a geodesic ray between two of them projecting down to a geodesic loop inside $C_i$. Such a loop cannot be isotopic to the boundary of $C_i$, thus contradicting the first hypothesis.

Since $\overline{C}_i$ has no cusp, no interior closed geodesic and only one geodesic boundary component, its fundamental group (which equals the fundamental group of $C_i$) must be trivial or cyclic generated by the translation about the geodesic boundary.

We deduce from this that $\overline{C}_i$ must be a hyperbolic half-plane or a funnel with geodesic boundary. Using the transport on geodesics orthogonal to $\partial S$ and basic Morse theory we see that all manifolds defined as $\{x;\dist(x,S)\leq r\},\,r>0$ are diffeomorphic to $S$ so their interiors are all diffeomorphic to $\Int(S)$. We deduce that $L$ is diffeomorphic to $\Int(S)$.
\end{proof}

\subsection{Noncompact surfaces}\label{ss.classifsurfaces}

\paragraph{ {\bf Ends of a space --}} Let us recall the definition of an end of a connected topological space $X$. Let $(K_n)_{n\in\N}$ be an exhausting and increasing sequence of compact subsets of $X$. An \emph{end} of $X$ is a decreasing sequence
$$\cC_1\supset\cC_2\supset...\supset\cC_n\supset...$$
where $\cC_n$ is a connected component of $X\moins K_n$. We denote by $\cE(X)$ the \emph{space of ends} of $X$. It is independent of the choice of $K_n$.

The space of ends of $X$ possesses a natural topology which makes it a compact subspace of a Cantor space. An open neighbourhood of an end $e=(\cC_n)_{n\in\N}$ is an open set $V\dans X$ such that $\cC_n\dans V$ for all but finitely many $n\in\N$.

\paragraph{ {\bf Classifying triples --}} In what follows, a \emph{classifying triple} is the data $\tau=(g,\cE_0,\cE)$ of
\begin{itemize}
\item a number $g\in\N\cup\{\infty\}$;
\item a pair of nested spaces $\cE_0\dans\cE$ where $\cE$ is a nonempty, totally disconnected and compact topological space; which satisfy
\item $g=\infty$ if and only if $\cE_0\neq\vide$.
\end{itemize}

Say that two classifying triples $\tau=(g,\cE_0,\cE)$ and $\tau'=(g',\cE_0',\cE')$ are \emph{equivalent} if $g=g'$ and if there exists a homeomorphism $h:\cE\to\cE'$ such that $h(\cE_0)=\cE_0'$.

\paragraph{{\bf Noncompact surfaces --}} We now recall the modern classification of surfaces as it appears in \cite{Ric}. The leaves of a hyperbolic surface laminations are orientable so we are only interested in the classification of \emph{orientable} surfaces.

Recall that an end $e=(\cC_n)_{n\in\N}$ of $\Sigma$ is \emph{accumulated by genus} if for every $n\in\N$, the surface $\cC_n$ has genus. The ends accumulated by genus form a compact subset that we denote by $\cE_0(\Sigma)\dans\cE(\Sigma)$. In our terminology the triple $\tau(\Sigma)=(g(\Sigma),\cE_0(\Sigma),\cE(\Sigma))$ is a classifying triple.

\begin{theorem}[Classification of surfaces]\label{classification}
Two orientable noncompact surfaces $\Sigma$ and $\Sigma'$ are homeomorphic if and only if their classifying triples $\tau(\Sigma)$ and $\tau(\Sigma')$ are equivalent.

Moreover for every classifying triple $\tau$ there exists an orientable noncompact surface $\Sigma$ such that $\tau(\Sigma)$ is equivalent to $\tau$.
\end{theorem}

\begin{remark}
As a direct consequence, there are uncountably many different topological types of open surfaces as there exists uncountably many closed subsets of the Cantor set.
\end{remark}

\subsection{Direct limits of surfaces}\label{ss.directlimit}

\paragraph{ {\bf Inclusions of surfaces --}} Let $S$ and $S'$ be two surfaces with boundary. We say that a map $f:S\to S'$ is an \emph{inclusion} if the two conditions below are satisfied.

\begin{itemize}
\item $f$ is continuous and injective.
\item $f$ maps every boundary component of $S$ to a boundary component of $S'$ or inside the interior of  $S'$.
\end{itemize}
When $\overline{S'\setminus f(S)}$ admits a hyperbolic structure with geodesic boundary we call $f$ a \emph{good inclusion}.

When we specify two points $x$ and $x'$ on $S$ and $S'$ respectively, a (good) inclusion $f:(S,x)\to(S',x')$ is supposed to map $x$ to $x'$.

\begin{remark}\label{rem_isom_embedding_proper_inclusion}
If $S$ and $S'$ are hyperbolic surfaces with geodesic boundary, then any isometric embedding $f:S\to S'$ is a good inclusion.
\end{remark}

\paragraph{ {\bf Direct limits --}} Let $(S_n)_{n\in\N}$ be a sequence of surfaces with boundary and a chain of inclusions $\{j_n:S_n\to S_{n+1}\}$. The \emph{direct limit} of this chain is the quotient
$$S_\infty=\underrightarrow{\lim}\{j_n:S_n\to S_{n+1}\}= \left.\bigsqcup S_n\right/\sim$$
where $\sim$ is the equivalence relation generated by $\forall x\in S_n,\,x\sim j_n(x)$. The space $S_\infty$ is naturally
 a topological surface, possibly with boundary. By definition of inclusions, a point of 
 $\partial S_\infty$ corresponds to a point $x$ of the boundary of some $S_{n_0}$ such that for every 
 $n\geq n_0$ the map $j_{n-1}\circ\ldots\circ j_{n_0}(x)$ belongs to the boundary of $S_n$. Moreover there exists an inclusion $J_n:S_n\to S_\infty$. 
 
 Direct limits enjoy the following universal property.

\begin{theorem}[Universal property]\label{t.universal_property}
Let $L$ be a surface. Assume that there exists a sequence of inclusions $\phi_n:S_n\to L$ which satisfy the compatibility condition
$$\phi_n=\phi_{n+1}\circ j_n.$$
Then there exists an inclusion $\phi:S_\infty\to L$ such that for every $n\in\N$
$$\phi_n=\phi\circ J_n.$$
\end{theorem}

\paragraph{ {\bf Open direct limits --}} Let $(S_n)_{n\in\N}$ be a sequence of surfaces with boundary and a chain of inclusions $\{j_n:S_n\to S_{n+1}\}$. The \emph{open direct limit} of this chain is by definition the interior of the direct limit. This is by definition an open surface.

\paragraph{ {\bf Geometric direct limits --}} We now assume that we are given a sequence $(S_n)_{n\in\N}$ of \emph{compact hyperbolic surfaces with geodesic boundary} and a chain of \emph{isometric embeddings} $\{j_n:S_n\to S_{n+1}\}$. This is in particular a chain of good inclusions (see Remark \ref{rem_isom_embedding_proper_inclusion}). The direct limit $S_\infty$ of this chain might not be complete: imagine the case of a sequence of surfaces obtained by gluing hyperbolic pairs of pants whose boundary components have length growing very fast. 

\begin{definition}
\label{d.geom_direct_limit}
The \emph{geometric direct limit} of the chain $\{j_n:S_n\to S_{n+1}\}$  of isometric embeddings of compact hyperbolic surfaces with geodesic boundary is the surface denoted by $\overline{S}_\infty$ and defined as the metric completion of the direct limit $S_\infty$.
\end{definition}

\begin{proposition}
\label{p:geom_limit}
The geometric direct limit of a chain $\{j_n:S_n\to S_{n+1}\}$ of isometric embeddings  of compact hyperbolic surfaces with geodesic boundary is a hyperbolic surface whose boundary components are disjoint geodesics (that can be closed or not) and satisfies that $\Int(\overline{S}_\infty)=\Int(S_\infty)$, the open direct limit.
\end{proposition}

\begin{proof}
An inductive argument using the uniformization theorem shows the following. There exist an increasing sequence of Fuchsian groups $\Gamma_1<\Gamma_2<\ldots\Gamma_n<\ldots$, an increasing sequence of connected domains with geodesic boundary of the Poincar\'e disk $\D$, denoted by $D_1\dans D_2\dans\ldots\dans D_n\dans\ldots$ (defined as the \emph{Nielsen cores} of the $\Gamma_n$, i.e. the convex hulls of their limits sets) and a sequence of isometries
$$\phi_n:S_n\to D_n/\Gamma_n$$
satisfying the following compatibility condition
$$\iota_n\circ\phi_n=\phi_{n+1}\circ j_n,$$
where $\iota_n:D_n/\Gamma_n\to D_{n+1}/\Gamma_{n+1}$ is the natural inclusion.

Let $D_\infty$ be the increasing union of the domains $D_n$ and $\Gamma_\infty$ be the increasing union of the groups $\Gamma_n$. The set $D_\infty$ is a convex set invariant by the Fuchsian group $\Gamma_\infty$ and $D_\infty/\Gamma$ is the direct limit of the inclusions $\iota_n:D_n/\Gamma_n\to D_{n+1}/\Gamma_{n+1}$. Using the universal property of direct limits (Theorem \ref{t.universal_property}) we see that $S_\infty$ is isometric to $D_\infty/\Gamma_\infty$.

The surface $D_\infty/\Gamma_\infty$ embeds isometrically inside the complete surface $\D/\Gamma_\infty$ so its metric completion, which is isometric to $\overline{S}_\infty$, is realized as the closure $\overline{D}_\infty/\Gamma_\infty$. The boundary of $D_\infty$ is geodesic (by convexity), its interior is precisely the interior of $D_\infty$, and these sets are $\Gamma_\infty$-invariant. This proves that the surface $\overline{D}_\infty/\Gamma_\infty$ satisfies the conclusion of the proposition, and hence $\overline{S}_\infty$ as well.
\end{proof}

\section{Illustrative examples}
\label{s.illustration}
\subsection{A lamination where every leaf is a disk} 

\paragraph{ {\bf Residual finiteness and geometry --}} Surface groups enjoy a property known as \emph{residual finiteness} (see \cite[III.18]{delaharpe} ). More precisely, given a closed hyperbolic surface $\Sigma$, and a non-trivial element $\gamma \in \pi_1(\Sigma)$ there exists a finite group $G$ and a morphism $\phi: \pi_1(\Sigma) \to G$ so that $\phi(\gamma)\neq 1$. In other words, every non-trivial element is disjoint from some finite index normal subgroup in $\pi_1(\Sigma)$. It is clear that this also implies that for any finite subset of $F \subset \pi_1(\Sigma) \setminus \{e\}$ there is a finite index normal subgroup $N \lhd \pi_1(\Sigma)$ such that $N \cap F = \emptyset$. From the geometric point of view this easily implies:

\begin{lemma}\label{l.res_large_rinj}
For every closed hyperbolic surface $\Sigma$ and every $K>0$ there exists a normal covering map $\hat p: \hat \Sigma \to \Sigma$ such that $\mathrm{sys}(\hat \Sigma) >K$. 
\end{lemma}

\begin{proof}
Just take as finite set the elements of $\pi_1(\Sigma)$ corresponding to simple closed geodesics of length $\leq K$ and the covering associated to the normal finite index subgroup given by residual finiteness which will then \emph{open} all short curves to give the desired statement. 
\end{proof}

Note that the covering constructed in Lemma \ref{l.res_large_rinj} has injectivity radius $> K$ at every point.

\paragraph{ {\bf Tower of coverings --}} Thus we can consider a tower of regular covering maps $\T=\{p_n:\Sigma_{n+1}\to\Sigma_n \}$ such that the injectivity radius at every point of $\Sigma_n$ tends to infinity with $n$. See Figure \ref{fig:Residouille}.

\begin{figure}[h!]
\centering
\includegraphics[scale=0.4]{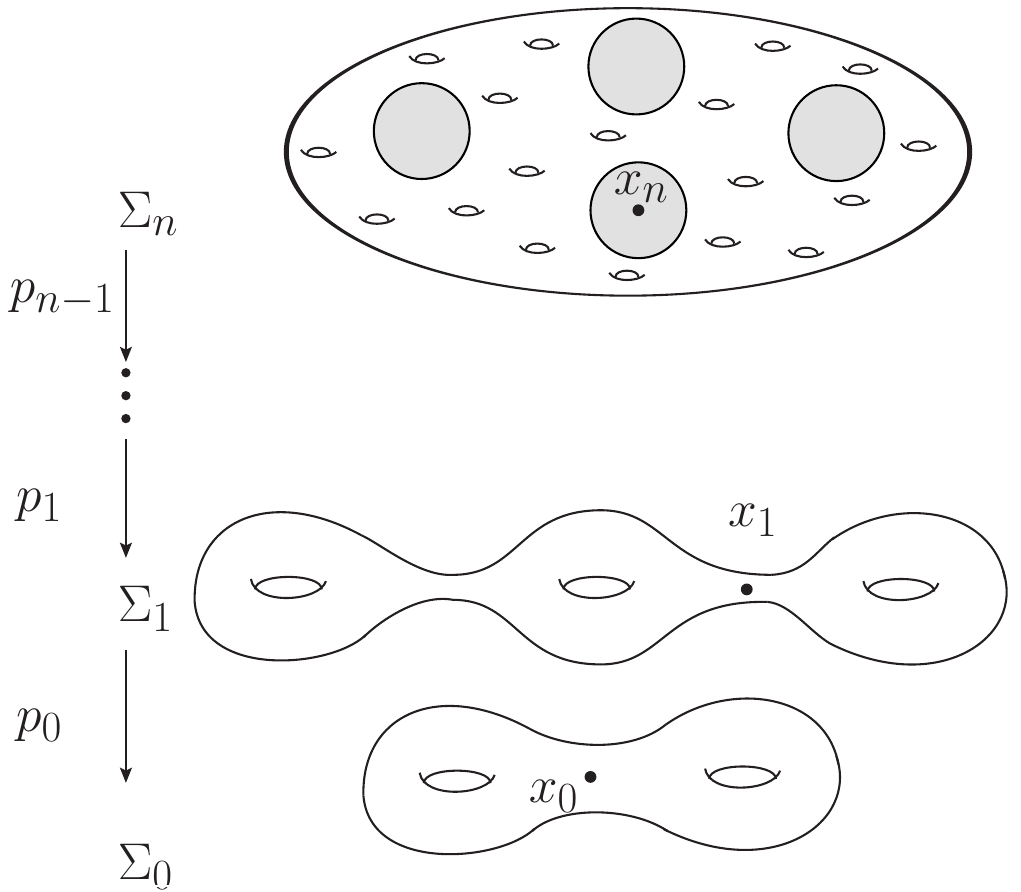}
\caption{Residual finiteness.}\label{fig:Residouille}
\end{figure}

The inverse limit of such a tower gives rise to a minimal lamination (see Proposition \ref{p.minimality}). Moreover Proposition \ref{p.Cheeger_Gromov} states that for every $K>0$ the $K$-neighbourhood of an element $\x\in \mathcal{L}$ inside its leaf is a copy of the $K$-neighbourhood of $x_n$ inside $\Sigma_n$ for $n$ large enough, which is an embedded disk of radius $R$. This means in particular that the leaf of every $\x$ is  an increasing union of disks: this must be a disk. This completes the construction of a minimal lamination all of whose leaves are hyperbolic disks. 

Notice that from the topological point of view, such examples are quite well known. For instance, one can take a totally irrational linear foliation in $\mathbb{T}^3$ to get a minimal foliation by planes. Also, the universal inverse limit construction, obtained as the inverse limit of \emph{all} finite coverings of a given surface gives rise to another minimal foliations by disks which is indeed the same as the one constructed above (as they are cofinal). We refer to \cite{Sullivan,Verjovsky}. See also \cite{Kapovich} for a realization of some laminations  by hyperbolic spaces as inverse limits of towers of coverings. From the point of view of our construction, nevertheless, this construction is quite illustrative, as it shows in the simplest possible context the strategy we want to follow to prove our main theorems, the key idea is to be able to construct a tower of regular coverings so that the injectivity radius grows in most places while we control that some other places get lifted carefully in order to get some leaves with topology. This will become clearer in our next example.

\subsection{Realizing a cylinder as a leaf} 
Let us show next how to construct a minimal lamination by Riemann surfaces for which one leaf is a cylinder, and every other leaf is a disk. This is not as simple as it seems and it contains one of the key difficulties in our whole construction. 

Using the same line of reasoning as above we would like to construct a tower of coverings such that for one sequence $\x=(x_n)_{n\in\N}$ of the inverse limits the $K$-neighbourhood of $x_n$ inside  $\Sigma_n$ is an embedded annulus when $n$ is large enough. The corresponding leaf would be an annulus. And furthermore,  when $\dist(x_n,y_n)\to\infty$ (so $\x$ and $\y$ lie on different leaves), we want the $K$-neighbourhood of $y_n$ to be a disk, when $n$ is large enough. The corresponding leaf would be a disk. See Figure \ref{fig:HW}.

\begin{figure}[h!]
\centering
\includegraphics[scale=0.4]{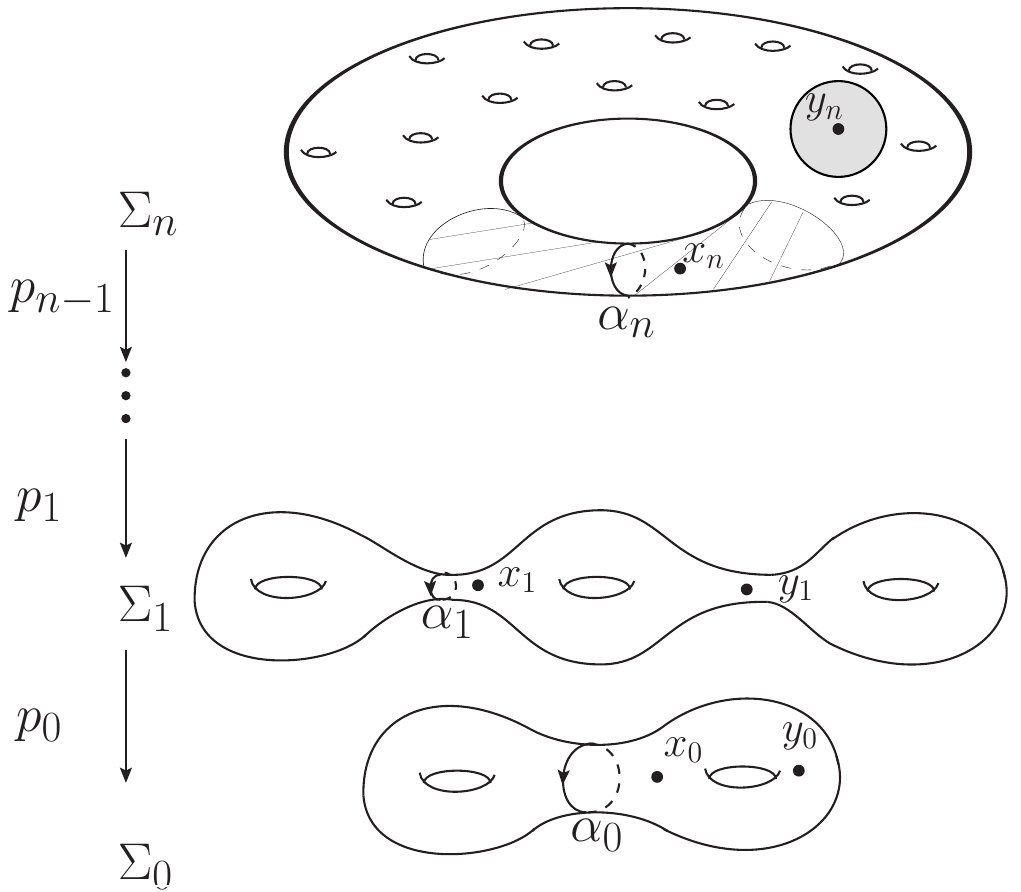}
\caption{Creating a cylinder.}\label{fig:HW}
\end{figure}

This is obtained by using our relative version of residual finiteness, Theorem \ref{teo.apendice}. For every simple closed geodesic $\alpha$ there exists a finite cover $\pi:\hat \Sigma\to\Sigma$ such that $\alpha$ has  a unique $(1:1)$-lift to $\hat \Sigma$, called $\hat \alpha$, and such that every other simple closed geodesic of $\hat \Sigma$ has length $\geq K$ where $K$ can be arbitrarily large. One could say that the \emph{second systole} of $\hat \Sigma$ is arbitrarily large.

Some basic facts of hyperbolic geometry (see Lemmas \ref{l.halfcollars} and \ref{l.inj_radius}) show that when $K$ is large enough, there is a collar about $\hat \alpha$ of width $\geq K/2$ and that points  away from $\hat \alpha$ have injectivity radii larger than $K$. This allows to implement the desired tower of finite coverings.

\begin{remark}\label{rem.obssuspS1}
It is possible to see that a representation $\rho: \pi_1(\Sigma) \to \mathrm{Homeo}(S^1)$ of  a surface group cannot produce, via the suspension construction, a foliation such that there is a unique annulus and the rest of the leaves are planes. This is because this would imply the existence of a non-abelian free subgroup\footnote{To see this it is enough to find two noncommuting elements in $\pi_1(\Sigma)$ which do not intersect the normal subgroup generated by an element which has a unique fixed point in $S^1$ (notice that only elements in the normal subgroup generated by this element can have fixed points since there is a unique non-planar leaf). To find such elements, one can look at the projection of $\pi_1(\Sigma)$ into its first homology group.} acting freely on $S^1$ which is impossible according to H\"older's Theorem \cite[Theorem 2.2.32]{Navas}. 
\end{remark}

More generally, one can show:

\begin{proposition}\label{rem.obstminimalfol}
Let $\mathcal{F}$ be a minimal\footnote{Sufficiently smooth, $C^{1+}$ is enough to use \cite{Alvarez_Yang}.} foliation by surfaces in a closed 3 manifold $M$ with all leaves of finite type and so that not every leaf is a disk, then it must have infinitely many leaves which are not disks. 
\end{proposition}
\begin{proof}
To see this, notice first that in this context there cannot be a transverse invariant measure: if a minimal foliation has a transverse invariant measure and one leaf is not a disk, then infinitely many leaves must have non-trivial fundamental group, notice that one can lift a non-trivial loop to nearby leaves, and these cannot become homotopically trivial in their leaves because of Novikov's theorem (recall that a minimal foliation cannot have a Reeb-component). 

Therefore Candel's theorem applies and there is a smooth Riemannian metric on $M$ such that leaves have negative curvature everywhere (see for example \cite[Theorem B]{Alvarez_Yang}). Thus, one can apply \cite[Theorem A]{Alvarez_Yang} to get a hyperbolic measure for the \emph{foliated geodesic flow} which produces an infinite number of periodic orbits and each corresponds to a non-trivial closed geodesic in some leaves. This is produced by a measure which has the SRB property (in particular its support is saturated by strong unstable manifolds) and therefore cannot be supported in finitely many leaves (see \cite[Proposition 3.1 (3)]{Alvarez_Yang}). 

If all leaves have finite topological type and are hyperbolic, since the injectivity radius must be bounded from below it follows that all closed geodesic in the leaves must lie in a compact core inside each leaf. In particular, if one looks at an accumulation point along the transverse direction in the support of the measure not all periodic orbits of the foliated geodesic flow can belong to the same leaf. 

This implies that there are infinitely many leaves with non-trivial topology. 
\end{proof}

It seems reasonable to expect the previous result to hold without any regularity assumption nor that assuming that all leaves have finite topological type, but we decided not to pursue this as it is not central for the results of this paper.  

\subsection{Realizing the Loch-Ness monster} 

Assume now that we wish to construct a minimal lamination for which one leaf is a Loch-Ness monster (i.e. has one end and infinite genus) and every other leaf is  a disk. In such an example, surfaces of finite and infinite topological type will coexist inside the same minimal hyperbolic surface lamination.

The strategy will be similar. We need to construct a tower of coverings $\mathbb{T} = \{ p_{n} : \Sigma_{n+1} \to \Sigma_n\}$ with the following properties (see Figure \ref{fig:lochnouille}): \begin{figure}[h!]
\centering
\includegraphics[scale=0.4]{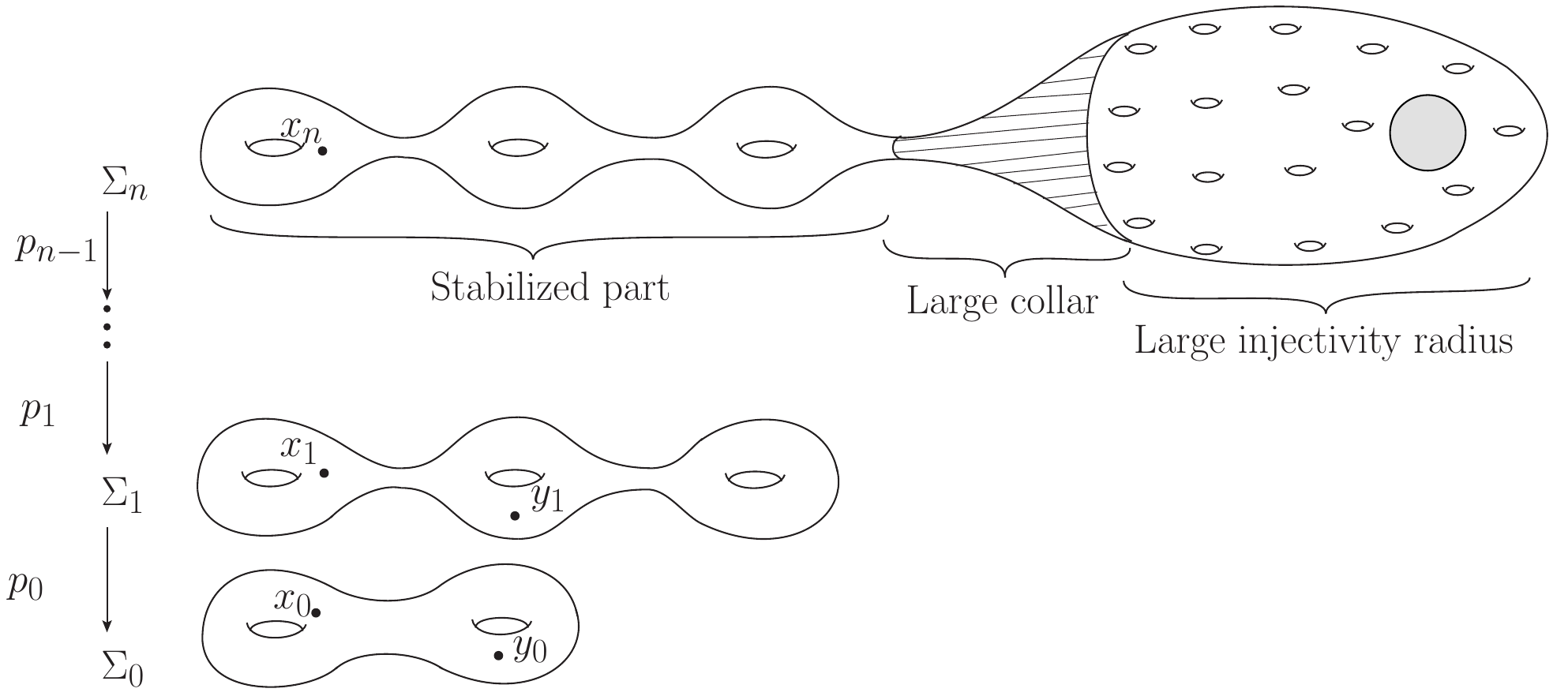}
\caption{Realizing a Loch-Ness monster.}\label{fig:lochnouille}
\end{figure}

\begin{enumerate}
\item\label{enum1} Each $\Sigma_n = S_n \cup X_n$ where $S_n$ and $X_n$ are subsurfaces with geodesic boundary and disjoint interiors (compare with admissible decompositions defined in \S \ref{ss.relative}). 
\item\label{enum2} The surface $S_n$ is connected, with genus $n$ and one boundary component. \emph{This is the part we want to stabilize.} Moreover, each surface $S_n$ admits a $(1:1)$-lift into a subsurface $S_n^\ast\subset S_{n+1}$. 
\item\label{enum3} The internal systole $\sigma_n$ of $\Sigma_n\setminus S_n^\ast$ (see Definition \ref{d.collars}) grows to $\infty$ with $n$.
\end{enumerate}

The fact that such a tower can be constructed relies on a strengthening of the relative residual finiteness mentioned in the previous section that is obtained in Theorem \ref{t.finitecovering}. A similar argument as above, provides the desired construction. 

Consider a sequence $\x=(x_n)_{n\in\N}\in\mathcal{L}$ so that $x_n\in S_n$ and $p_{n}(x_{n+1})=x_{n}$ for every $n\in\N$. In this case the leaf through $\x$ is a Loch-Ness monster (this follows from Proposition \ref{p:topology_leaves}). On the other hand, the injectivity radius over any sequence $\y=(y_n)_{n\in\N}\in\cL$ such that $\dist(x_n,y_n)$ is unbounded, goes to infinity with $n$ (see Lemma \ref{l:discs}). This implies that every leaf different from that containing $\x$ is a disk.

\subsection{Combining both examples and some comments on cylinders}\label{ss.cylinders} 

It can be noticed from the proof of Theorem \ref{t.finitecovering} below that it is possible to adapt the relative residual finiteness in order to consider a tower of coverings $\mathbb{T}= \{ p_n : \Sigma_{n+1} \to \Sigma_n \}$ so that:

\begin{itemize}
\item Each $\Sigma_n = S_n \cup X_n$ where $S_n$ and $X_n$ are subsurfaces with geodesic boundary and disjoint interiors. 
\item The surface $S_n$ is connected, with genus $n$ and one boundary component and admits a $(1:1$)-lift into a subsurface  $S_n^\ast\subset	S_{n+1}$.
\item The surface $X_n$ has the property that it contains a unique simple closed geodesic $\alpha_n$ which has length smaller than $1$ but every other primitive closed geodesic of length smaller than $n$ has to be homotopic to either the boundary of $X_n$ or to $\alpha_n$ (we could say that the \emph{second internal systole} grows to infinity).  
\item The distance between $\partial S_n$ and $\alpha_n$ goes to infinity with $n$.
\end{itemize}

Notice in particular that $p_n$ must map $\alpha_{n+1}$ as a (1:1) covering of $\alpha_n$ for sufficiently large $n$. As in the previous examples, this construction will produce a leaf homeomorphic to a Loch-Ness monster corresponding to the sequence of subsurfaces $S_n$, and a leaf homeomorphic to an annulus corresponding to the sequence $\alpha_n$. Moreover, it can be shown that any other leaf is homeomorphic to a disk (this follows again from Lemma \ref{l:discs}).

In what follows we will extend this constructions in order to be able to produce several different possible laminations. As this example shows, the production of cylinders in the lamination is a bit different from the construction of the Loch-Ness monster as one requires the construction of surfaces in the coverings, and cylinders are detected by closed geodesics with large collar neighbourhoods. It turns out that a procedure similar to the one used to construct the Loch-Ness monster works for \emph{every} other surface (except the cylinder). It is possible to find a more cumbersome formalism that includes cylinders, but  in order to simplify the presentation, we will ignore cylinder leaves and leave the construction of laminations which also have cylinder leaves to the reader. 

\section{Toolbox for constructing finite coverings}\label{s.toolbox}

We now give the principal tool that we will use in order to implement the idea given in \S \ref{s.illustration}. This is a variation of the residual finiteness of surface groups.

\subsection{A relative version of residual finiteness}\label{ss.relative}

\paragraph{ {\bf The key tool --}} The following result will provide us with an essential tool for the proof of Theorem \ref{t.finitecovering} and will be used in several points.  Its proof is deferred to Appendix \ref{s.appendix}.

\begin{theoremc}\label{teo-HW}

Let $\Sigma$ be a closed hyperbolic surface, and let $\alpha\subset\Sigma$
  be a simple closed geodesic. Then, for all $K>0$, there exists a finite 
  covering $\pi:\hat \Sigma\to\Sigma$ such that
  \begin{itemize}
  \item $\hat \Sigma$ contains a non-separating simple closed geodesic such that $\pi(\hat \alpha)=\alpha$ and $\pi$ restricts to a homeomorphism on $\hat \alpha$;
  \item every simple closed geodesic which is not $\hat \alpha$ has length larger than $K$.
  \end{itemize}
\end{theoremc}

We say that the \emph{second systole} of $\hat \Sigma$ is large because the only short closed curves (i.e. shorter than $K$) in $\hat \Sigma$ need to be homotopic to a power of $\hat \alpha$. 

\begin{remark}
\label{rem.large_collars_tubes}
By Lemma \ref{l.halfcollars} the surface $\hat \Sigma$ has half-collars around $\hat \alpha$ of width $\geq \frac{K-l_\alpha}{2}$ on both sides. If $K\geq2 l_\alpha$ then the widths of these half-collars are  $\geq K/4$.
\end{remark}

\begin{remark}
\label{rem_tubes_genus}
The internal systole of the connected surface with boundary $T$ obtained by cutting $\hat \Sigma$ along $\hat \alpha$ is greater than $K$ so there must exist a point of $T$ with injectivity radius $\geq K$. We deduce that the area of $T$ (which equals that of $\hat \Sigma$) is $\geq 2\pi(\cosh K-1).$ By Gauss-Bonnet's theorem, the genus $g$ of $T$ (note that it equals $genus(\hat \Sigma)-1$) satisfies the following inequality
$$g\geq\frac{\cosh K -1}{2}.$$
\end{remark}

\paragraph{ {\bf Admissible decompositions --}} Let $\Sigma$ be a closed hyperbolic surface. An \emph{admissible decomposition} of $\Sigma$ is a pair $(X,S)$ of (possibly disconnected) compact hyperbolic subsurfaces of $\Sigma$ with geodesic boundary such that
\begin{itemize}
\item $\Sigma=X\cup S$;
\item $X$ and $S$ meet at their common boundary.
\end{itemize}

The surface $X$ will be sometimes called the \emph{admissible complement of $S$}.

\paragraph{ {\bf Relative residual finiteness --}} We now state a relative version of residual finiteness of the fundamental group of a given closed hyperbolic surface $\Sigma$ which is adapted to a given admissible decomposition $(X,S)$. More precisely, we want to find coverings of $\Sigma$ where we keep a copy of $S$ while increasing the internal systole and collar width of its admissible complement $X$. The following theorem in the case where $S$ is empty, can be deduced from the residual finiteness of surface groups. 

\begin{theorem}[Relative residual finiteness]\label{t.finitecovering}
Let $\Sigma$ be a closed hyperbolic surface and $(X,S)$ be an admissible decomposition of $\Sigma$. Then, given $K>0$, there exists a closed hyperbolic surface $\hat \Sigma$ with an admissible decomposition $(\hat X,\hat S)$  as well as a finite covering $p: \hat \Sigma \to \Sigma$ such that
\begin{enumerate}
\item the restriction $p|_{\hat S}:\hat S\to S$ is a $(1:1)$ isometry;
\item the internal systole of $\hat X$ is larger than $K$;
\item the boundary components of $\hat X$ have disjoint half collars of width larger than $K$.
\end{enumerate}
\end{theorem}
Notice that conditions (1) and (2) imply condition (3) with a smaller constant depending on the length of the boundary components of $S$. We state the three conditions because this is the way we will use it.

\begin{remark}[Enough topological room on $\hat X$]\label{r.room} Assume the hypothesis of Theorem \ref{t.finitecovering}, and consider $S_0$ a compact surface with boundary (not necessarily connected and possibly with degenerate components\footnote{This would be isolated simple closed geodesics with large embedded collars. This is needed if one wishes to construct cylinder leaves, but as we mentioned before, we will ignore this to avoid cumbersome notation.}) and $g\in\N$. Then, we can perform the construction so that, in addition to the conclusion of Theorem \ref{t.finitecovering}, $\hat X$ contains a subsurface $X$ with geodesic boundary, written as a disjoint union $X=X_1\sqcup X_2$ such that:
\begin{itemize}
\item $X_1$ is homeomorphic to $S_0$
\item for every boundary component $\alpha_j$ of $\hat X$ there exists a subsurface $A_j\dans X_2$ such that
\begin{enumerate}
\item $A_j$ has genus $g$ and two boundary components;
\item one of the boundary components of $A_j$ is $\alpha_j$;
\item for $i\neq j$ we have $A_i\cap A_j=\vide$.
\end{enumerate}
\end{itemize}
This means that boundary components of $\hat X$ are separated by topology.
\end{remark}

The rest of the section is devoted to the proof of Theorem \ref{t.finitecovering} and Remark \ref{r.room}. For that purpose we will show how to perform surgeries of finite coverings.

\subsection{Surgeries of finite coverings}

We now introduce a technique to construct new finite coverings of a surface from others. We call this technique \emph{surgery} since it consists in `cutting and pasting' different finite coverings (see Figure \ref{fig:surg}).

\begin{figure}[h!]
\centering
\includegraphics[scale=0.5]{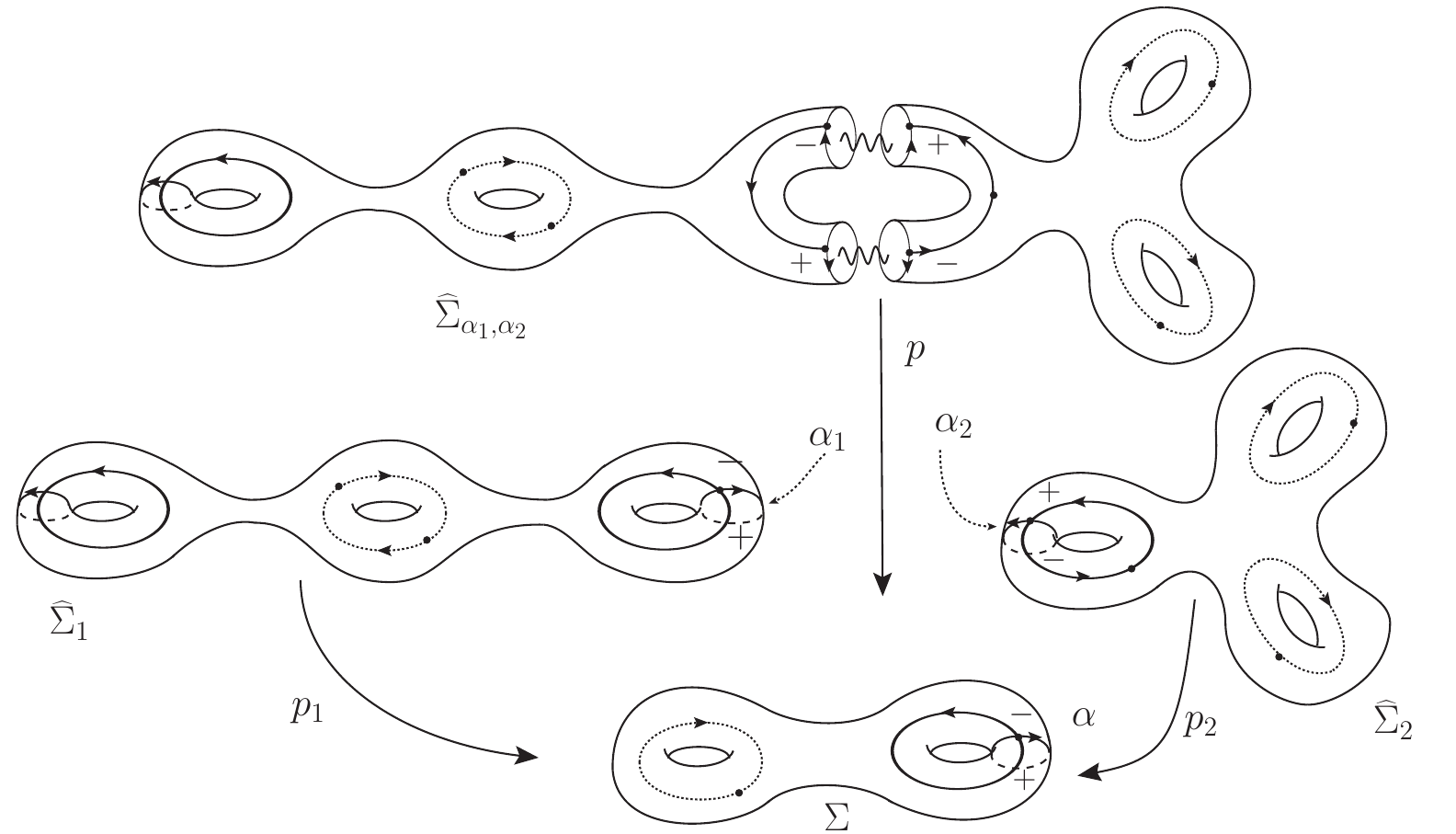}
\caption{The surgery of two double covers.}\label{fig:surg}
\end{figure}

Given a closed hyperbolic surface $\Sigma$ and a simple closed geodesic $\alpha \subset \Sigma$ we denote by $\Sigma_\alpha$ the (not necessarily connected) hyperbolic surface obtained by cutting along $\alpha$. Two new boundary components appear in $\Sigma_\alpha$ associated to $\alpha$ which we denote by $\alpha^+$ and $\alpha^-$ according to the orientation. We will fix a point $x\in\Sigma$ and its two copies $x^+\in\alpha^+$ and $x^-\in\alpha^-$.

Let $p_i: \hat \Sigma_i \to \Sigma$ be coverings of $\Sigma$ ($i=1,2$) such that for some simple closed geodesic $\alpha$ of $\Sigma$ there exists, for every $i$, a closed geodesic $\alpha_i \subset \hat \Sigma_i$ such that $p_i$ is a $(1:1)$ isometry from $\alpha_i$ to $\alpha$.  Consider the two surfaces $\hat \Sigma_{\alpha_i}$, the four copies $\alpha_i^\pm$ of $\alpha$ as well as the four points $x^\pm_i$, which project down to $x$.

We define $\hat \Sigma_{\alpha_1,\alpha_2}$ to be the surface obtained from $\hat \Sigma_{\alpha_1}$ and $\hat \Sigma_{\alpha_2}$ by gluing $\alpha_1^+$ with $\alpha_2^-$ and $\alpha_1^-$ with $\alpha_2^+$.  To completely describe the gluing one must require that it sends respectively $x_1^+$ and $x_1^-$ on $x_2^-$ and $x_2^+$ and that it is an isometry: notice that the four curves are isometric lifts of $\alpha$.

Together with $\hat \Sigma_{\alpha_1,\alpha_2}$ one can define a map $p: \hat \Sigma_{\alpha_1,\alpha_2} \to \Sigma$ which is obtained by applying $p_1$ or $p_2$. We let $\hat \alpha^+$ and $\hat \alpha^-$ be these two distinguished curves, which are both isometric to $\alpha$

\begin{definition}
The map $p: \hat \Sigma_{\alpha_1, \alpha_2} \to \Sigma$ is called the \emph{surgery} of $p_1$ and $p_2$ \emph{along the pair} $(\alpha_1,\alpha_2)$. 
\end{definition}

\begin{proposition}
\label{propo_connected_covering}
The map $p$ is a finite cover of $\Sigma$ and the surface $\hat \Sigma_{\alpha_1, \alpha_2}$ is connected if the $\hat \Sigma_i$ are connected and at least one of the $\alpha_i$ is non-separating. 
\end{proposition}

\begin{proof}
Suppose $\hat \Sigma_i$ are connected and $\alpha_1$ is non separating. Then $\hat \Sigma_{\alpha_1,\alpha_2}$ is a union of  $\hat \Sigma_{\alpha_1}$, a connected surface, with two (if $\alpha_2$ separates) or one (if it does not) connected surfaces intersecting $\hat \Sigma_{\alpha_1}$. It must be connected.

Moreover the surface $\hat \Sigma_{\alpha_1,\alpha_2}$ is compact so it suffices to prove that $p$ is a local isometry. Since $p_1$ and $p_2$ are local isometries, it is enough to verify this in a neighbourhood of $\hat \alpha^+$ and $\hat \alpha^-$.

Recall that a point $y$ inside a sufficiently thin collar in $\Sigma$ about $\alpha$ is described by its \emph{Fermi coordinates} (based at $x$) $(\rho(y),\theta(y))$ where $\rho(y)$ is the signed distance from $y$ to $\alpha$ (positive on the right side, negative on the left side) and where the orthogonal projection of $y$ onto $\alpha$ is $\alpha(\theta(y))$ (here we choose $\alpha(0)=x$). See \cite{Bu}.

Since they are isometries from a collar of $\alpha_i$ onto a collar onto $\alpha$, the maps $p_i$ preserve Fermi coordinates (based at $x_i$ respectively). So by construction the map $p$ preserves Fermi coordinates in a collar of $\hat \alpha^\pm$. This means that it is an isometry from these open sets onto a collar about $\alpha$, concluding the proof.
\end{proof}

 Finally, note that by construction there are two isometric embeddings $j_1:\hat\Sigma_{\alpha_1}\to\hat\Sigma$ and $j_2:\hat\Sigma_{\alpha_2}\to\hat\Sigma$ such that for $i=1,2$
$$p\circ j_i=p_i$$
when restricted to $\Int(\hat \Sigma_{\alpha_i})$. We will refer to them as the two \emph{$(1:1)$-lifts} of $\hat\Sigma_{\alpha_1}$ and $\hat\Sigma_{\alpha_2}$ to $\hat\Sigma$.

\subsection{ Attaching tubes}

To construct the desired coverings we will use Theorem C several times and perform surgeries from this. To simplify the structure we will give a name to the building blocks of the surgeries provided by Theorem C. 

\begin{definition}
Given $\alpha \subset \Sigma$ a simple closed geodesic and $K>0$ we say that $T$ is an $(\alpha,K)$-tube if it is a surface given by Theorem C for the curve $\alpha$ and the constant $K$.
\end{definition}
Notice that an $(\alpha,K)$-tube is not topologically a simple surface, like the word `tube' might suggest. In fact, since it covers some closed hyperbolic surface, it must be itself a surface of hyperbolic type, and in our applications it will usually have large genus (see Remark \ref{rem_tubes_genus}).

Given an admissible decomposition $(X,S)$ of $\Sigma$ we will ``attach tubes'' to boundary components of $S$ in order to isolate them one from the other.

\paragraph{{\bf Attaching tubes at a closed geodesic --}} Given a surface $\Sigma$ with a simple closed geodesic $\alpha$. Consider $T$, a $(\alpha,K)$-tube, and $\pi:T\to\Sigma$ the covering map defined by Theorem C and $\hat \alpha$, the unique $(1:1)$-lift of $\alpha$.

\begin{figure}[h!]
\centering
\includegraphics[scale=0.4]{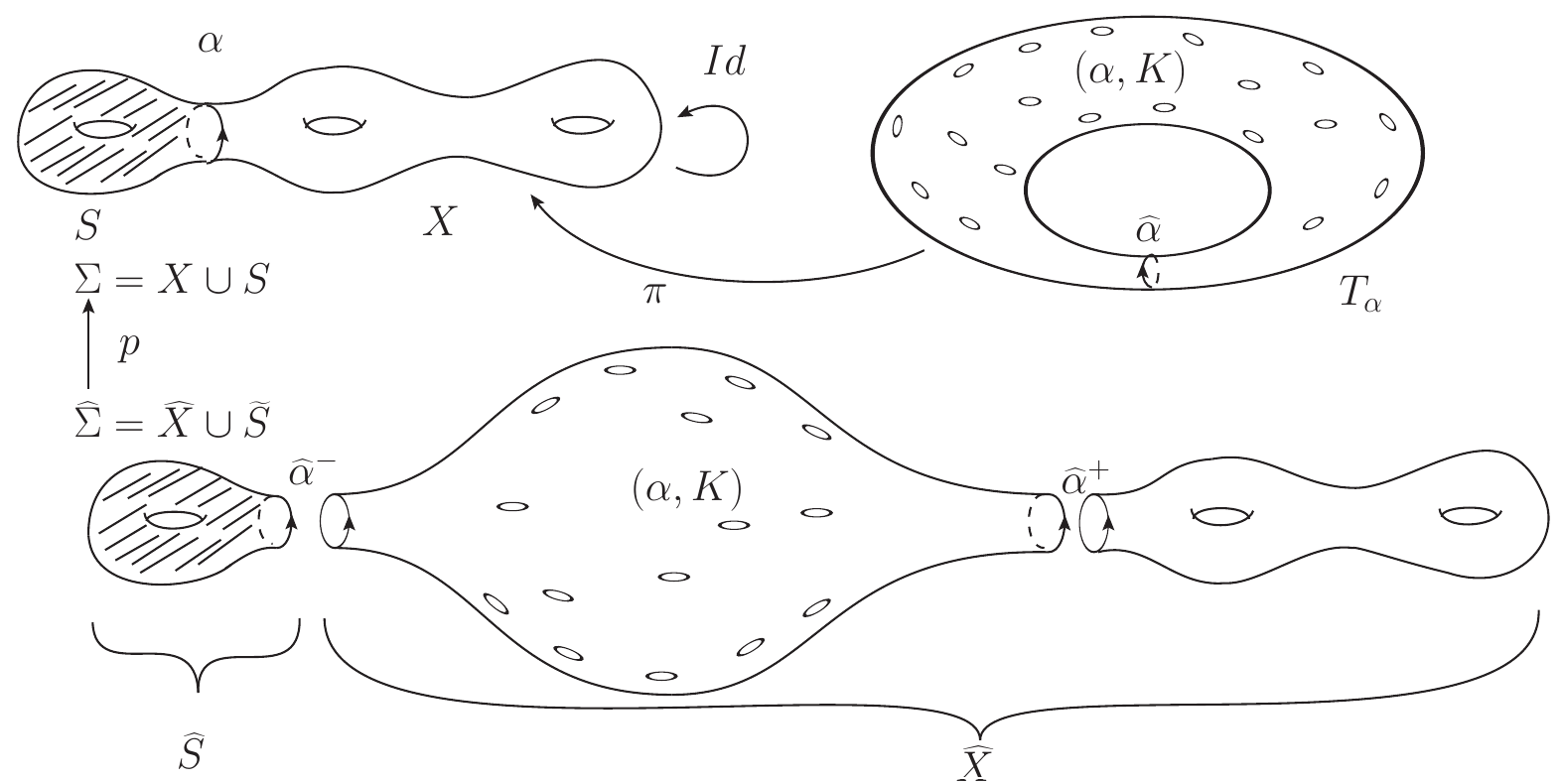}
\caption{Attaching a tube.}\label{fig:tubouille}
\end{figure}

We will say that a covering $p: \hat \Sigma \to \Sigma$ is obtained from $\Sigma$ by \emph{attaching a $(\alpha,K)$-tube} at $\alpha$ if it is the surgery of $\pi$ and the identity $\Id: \Sigma \to \Sigma$ along the pair $(\hat \alpha,\alpha)$. Since the curve $\hat \alpha$ is non-separating, the surface $\hat \Sigma$ is connected (see Proposition \ref{propo_connected_covering}).

In $\hat \Sigma$ there are two distinguished curves $\hat \alpha^+$ and $\hat \alpha^-$ which are the unique $(1:1)$-lifts of $\alpha$ and have at least one half collar of width $\geq \frac{K-l_\alpha}{2}\geq K/4$ (if $K\geq 2l_\alpha$). See Figure \ref{fig:tubouille}.

Let $j_1,j_2$ be the two $(1:1)$-lifts associated to $p:\hat \Sigma\to\Sigma$ and note that by definition $p\circ j_1=\Id$ when restricted to $\Int(\Sigma_\alpha)$.

\begin{lemma}\label{l_short_curve}
Let $\Sigma$ be a closed surface, $\alpha$ be a simple closed geodesic and $K\geq 2l_\alpha$. Let $p:\hat\Sigma\to\Sigma$ be a finite covering obtained from $\Sigma$ by attaching a $(\alpha,K)$-tube. Let $\beta\dans\hat\Sigma$ be a closed geodesic of length $< K/4$. Then $\beta$ is included inside $j_1(\Int(\Sigma_\alpha))$.
\end{lemma}

\begin{proof}
Let $\beta\dans\hat\Sigma$ be a closed geodesic. There are two possibilities

\noindent {\it Case 1. $\beta$ is disjoint from $\hat \alpha^+$ and $\hat\alpha^-$.} In that case either $\beta$ is included inside $j_1(\Int(\Sigma_\alpha))$ or inside $j_2(\Int(T_\alpha))$. In the second case $l_\beta$ is larger that the second systole of the tube, so it is $\geq K$

\noindent {\it Case 2. $\beta$ crosses $\hat \alpha^+$ or $\hat \alpha^-$.} In that case its length must be larger than the width of a half-collar based at $\hat \alpha^+$ or at $\hat \alpha^-$, which is $\geq K/4$.

This proves that if furthermore $l_\beta< K/4$ then $\beta\dans j_1(\Int(\Sigma_\alpha))$.
\end{proof}

\subsection{Proof of Theorem \ref{t.finitecovering}.}

The proof of Theorem \ref{t.finitecovering} consists in starting with $\Sigma$ and attaching several $(\alpha,K)$-tubes. The construction has several stages. Consider a closed hyperbolic surface $\Sigma$ with admissible decomposition $(X,S)$. Let $\alpha_1,\ldots,\alpha_k$ denote the boundary components of $S\dans\Sigma$. Let $K>0$ and $L>4K$ satisfying moreover $L\geq 2l_{\alpha_i}$ for all $i$.

\begin{remark}\label{r.collars_admissible}
Let $1\leq i\leq k$. If $\alpha_i$ is a boundary component of $S$ then it has one half-collar included inside $S$ and another one included in $X$.

\end{remark}

\paragraph{ {\bf Isolating components of $S$ --}}

Denote $\Sigma_0= \Sigma$ and consider a finite cover $p_1: \Sigma_1 \to \Sigma_0$ obtained by attaching an $(\alpha_1,L)$-tube along $\alpha_1$.

There is exactly one $(1:1)$-lift of $\alpha_j$ for $j\geq 2$, and exactly two $(1:1)$-lifts of $\alpha_1$. Moreover only one of these two lifts has a half-collar that projects down into $S$ (see Remark \ref{r.collars_admissible}). This lift bounds a $(1:1)$-copy of the corresponding connected component of $S$ and has one half collar of width more than $L/4\geq K$. Let $\hat \alpha_1$ denote the lift of $\alpha_1$ that we distinguished.

As a consequence, the surface $\Sigma_1$ possesses an admissible decomposition $(X_1,S_1)$ $p_1|_{S_1}:S_1\to S$ is an isometry and where $\hat \alpha_1$ has one half-collar of width $\geq K$ included in $X_1$.

We can continue this process and construct $p_j : \Sigma_j \to \Sigma_{j-1}$ for $j = 2, \ldots, k$ by attaching $(\alpha_j,L)$-tubes along $\alpha_j \subset \Sigma_{j-1}$. This produces a finite cover $\hat p_k: \Sigma_k \to \Sigma$ where $\Sigma_k$ has an admissible decomposition $\Sigma_k=X_k\cup S_k$ such that 

\begin{itemize}
\item each component of $S$ has a $(1:1)$-lift to $S_k$ by $\hat p_k$;
\item boundary components of $X_k$ have disjoint half-collars of width $\geq K$.
\end{itemize}

\begin{remark}
\label{rem_no-pants}
Tubes have large genus by Remark \ref{rem_tubes_genus} so we can assume that $X_k$ is a finite union of compact connected surfaces with boundary, none of which is a pair of pants. 
\end{remark}


\paragraph{ {\bf Enlarging the internal systole of $X$ --}}  To complete the proof we need to take a finite cover that lifts $S_k$ while enlarging the internal systole of $X_k$, the admissible complement of $S_k$. This will be done by attaching tubes to large simple closed geodesics intersecting those curves of $X_k$ that have length $<K$.

\begin{proposition}\label{p.fill_large_curves}
Let $Y$ be a connected compact hyperbolic surface with geodesic boundary which is not a pair of pants and let $\beta$ be a closed geodesic inside the interior of $Y$. For every $L>0$ there exists a simple closed geodesic $\gamma$ such that $\beta\cap\gamma\neq\vide$ and $l_{\gamma}\geq L$ for $i=1,2$.
\end{proposition}

\begin{proof}
The set $Y$ is not a pair of pants so there exists a filling pair $(\gamma_1,\gamma_2)$ of simple closed geodesics, meaning that $Y\moins(\gamma_1\cup\gamma_2)$ is a union of disks and annuli isotopic to the boundary of $Y$ (see \cite[Proposition 3.5.]{Farb_Marg}). In particular every closed geodesic inside the interior of $Y$ meets the union $\gamma_1\cup\gamma_2$ and $i(\gamma_1,\gamma_2)>0$.

Let $\gamma_1'$ be the simple closed geodesic obtained from $\gamma_1$ after performing a large enough number of Dehn twists about $\gamma_2$ so that $l_{\gamma_1'}\geq L$. We have $i(\gamma_1',\gamma_2)>0$ so we can also obtain from $\gamma_2$ a simple closed geodesic $\gamma_2'$ after iterating a large enough number of Dehn twists about $\gamma_1'$ so that $l_{\gamma_2'}\geq L$. By construction, the pair of curves $(\gamma_1',\gamma_2')$ remains filling. In particular every closed geodesic $\beta$  inside the interior of $Y$ meets $\gamma_1'$ or $\gamma_2'$. Since each of these curves has length $\geq L$ this ends the proof of the lemma.
\end{proof}

\begin{proof}[End of the proof of Theorem \ref{t.finitecovering}.]

Since $X_k$ is a compact hyperbolic manifold its length spectrum (c.f. Appendix \ref{s.appendix}) is discrete and there are finitely many closed geodesics (not necessarily simple) $\beta_1,\ldots,\beta_l\dans X_k$ with length $< K$ and that are inside the interior of $X_k$.

Consider first the curve $\beta_1$. The connected component of $X_k$ containing $\beta_1$ is not a pair of pants by Remark \ref{rem_no-pants}. Hence we can apply Proposition \ref{p.fill_large_curves} and find a simple closed curve $\gamma$ with length greater than $L=4K$. Let us rename $\hat \Sigma_0=\Sigma_k$, $\hat S_0=S_k$ and $\hat X_0= X_k$. Consider the covering map $\hat p_1: \hat \Sigma_1 \to \hat \Sigma_0$ defined by attaching a $(\gamma, L)$-tube along $\gamma$.  Each connected component of $\hat S_0$ has a unique $(1:1)$-lift to $\hat \Sigma_1$, this defines an admissible decomposition $(\hat X_1,\hat S_1)$ of $\hat \Sigma_1$. The surface $\hat X_1$ has half-collars of width $\geq K$, and its boundary components are separated by surfaces of high genus. 

Moreover, by Lemma \ref{l_short_curve}, the only closed geodesics of $\hat \Sigma_1$ of length $< K$ are included inside $j_1(\Int([\hat \Sigma_0]_\gamma))$. In particular the only closed geodesics inside $\Int(\hat X_1)$ with length $< K$ are the lifts $j_1(\beta_i)$ of the curves $\beta_i\dans\hat X_0$ satisfying $\beta_i\cap\gamma=\vide$. So after this step, the number of closed geodesics of length $<K$ inside the admissible complement is strictly smaller. Reasoning inductively we obtain the desired cover $\hat p:\hat \Sigma\to\Sigma$ where the admissible decomposition $\hat \Sigma=\hat X\cup\hat \Sigma$ satisfies the three Items of Theorem \ref{t.finitecovering}.

Note that since tubes have large genus (see Remark \ref{rem_tubes_genus}) we can, up to attaching tubes at large curves inside $\hat X$, ensure that the conclusion of Remark \ref{r.room} holds: there is enough topological room inside $\hat X$.
\end{proof}

\section{Forests of surfaces and towers of finite coverings}\label{s.admissible}
\subsection{Organizing surfaces in forests}\label{ss.surfaces-forests}

We  now use the combinatorial description of surfaces and the concept of open direct limit to organize a family of open surfaces. This seemingly complicated way to organize the surfaces gives us more flexibility to control the topology of the leaves of a lamination constructed as a tower of coverings (see Remark \ref{r.forestjustif}). 

\paragraph{ {\bf Forests --}} A forest will be defined as a countable union of disjoint rooted trees. Let us  be more precise and state some notations.

Let $G=(V,E)$ be an oriented graph where $V$ is the set of vertices of $G$ and $E\subset V^2$ is the set of edges. We define the \emph{origin} and \emph{terminal} functions $o:E\to V$ and $t:E\to V$ so that $e=(o(e),t(e))$ for every $e\in E$. 

\begin{definition}\label{d.forest} A \emph{forest} is an oriented graph $\mathcal{T}=(V(\mathcal{T}),E(\mathcal{T}))$ where the set $V(\cT)$ of vertices and the set $E(\cT)\dans V(\cT)^2$ of oriented edges satisfy

 \begin{itemize}
 \item The set of vertices $V(\mathcal{T})$ has a countable partition $V(\mathcal{T})=\bigsqcup_{n\in\N}V_n(\mathcal{T})$ were the $V_n(\mathcal{T})$ are finite sets. We call $V_n(\mathcal{T})$ the $n$-th floor of $\mathcal{T}$. 
 \item $E(\mathcal{T})$ is contained in $\bigcup_{n\in\N} (V_n(\mathcal{T})\times V_{n+1}(\mathcal{T}))$. In other words, given any edge, its terminal vertex is one floor above its origin vertex.
\item Every vertex is the terminal vertex of at most one edge. This implies that $\mathcal{T}$ has no cycles.
\item Every vertex is the origin vertex of at least one edge.
\end{itemize}

\end{definition}

We will write $E(\mathcal{T})=\bigsqcup_{n\in\N}E_n(\mathcal{T})$ where $E_n(\mathcal{T})=\{e\in E(\mathcal{T}):o(e)\in V_n(\mathcal{T})\}$. 

A \emph{root} of $\mathcal{T}$ is a vertex $v\in V(\mathcal{T})$ that is not the terminal vertex of any edge: a root can be located at an arbitrary level. We note $R(\mathcal{T})$ the set of roots of $\mathcal{T}$ and $R_k(\mathcal{T})$ the set of roots of $\mathcal{T}$ that belong to $V_k(\mathcal{T})$. Notice that $\mathcal{T}=\bigsqcup_{v\in R(\mathcal{T})} \mathcal{T}_{v}$ where $\mathcal{T}_{v}$ is the maximal connected subtree of $\mathcal{T}$ containing the root $v$. 

On the other hand, we define the \emph{ends} of $\mathcal{T}$ as the union of the ends of its sub-trees, that is $$\cE(\mathcal{T})=\bigsqcup_{v\in R(\mathcal{T})} \cE(\mathcal{T}_{v})$$
A \emph{ray} of $\mathcal{T}$ is a concatenation of edges $r=(e_n)$ starting at a root. We will index those edges according to the floor to which they belong, that is: if a ray $r$ starts at a root $v\in R_k(\mathcal{T})$, we will denote its edges as $e_k e_{k+1}\ldots$ Also, we will consider $r_{\alpha}$ as a graph morphism $r_{\alpha}:[k,+\infty)\to\mathcal{T}$ where $[k,+\infty)$ is the half-line with one vertex for each integer greater or equal than $k$ and $r_{\alpha}([n,n+1])\in E_n(\mathcal{T})$.

Notice that the set of rays is in correspondence with $\cE(\mathcal{T})$, we will note $r_{\alpha}$ the ray converging to $\alpha$. 
%

\paragraph{ {\bf Forests of surfaces --}} A \emph{forest of surfaces} is a triple 
$$\mathcal{S}=(\mathcal{T},\{S_v\}_{v\in V(\mathcal{T})},\{j_e\}_{e\in E(\mathcal{T})})$$ 
where $\mathcal{T}$ is a forest , $\{S_v\}_{v\in V(\mathcal{T})}$ is a family of pointed compact and connected surfaces with boundary and $\{j_e\}_{e\in E(\mathcal{T})}$ is a family of good inclusions $j_e:S_{o(e)}\to S_{t(e)}$. When necessary, we will note the pointed surface $(S_v,q_v)$, however we will omit the pointing whenever it is possible. 

We associate to $\mathcal{S}$ a  family of pointed surfaces $\{\Ss^{\alpha}\}_{\alpha\in\cE(\mathcal{T})}$  that is called the set of \emph{limit surfaces} of $\mathcal{S}$ and is defined as follows. For an end $\alpha\in\cE(\mathcal{T})$ with the corresponding ray $r_{\alpha}=(e_n)_{n\geq k}$ ($k$ being the floor of the corresponding root) and the chain of inclusions $\{j_{e_{n}}:S_{o(e_n)}\to S_{t(e_n)}\}_{n\geq k}$ associated to it. We define $\Ss^\alpha$ as the open direct limit of this chain.

As mentioned in \S \ref{s.illustration}, we need to be very careful in our construction of towers of coverings if we want to control the topology of leaves in the inverse limit. The next definition gives the correct way to organize the towers.

\subsection{Admissible towers and forests}

\paragraph{ {\bf Forests of surfaces included in towers --}} A forest of surfaces $$\mathcal{S}=(\mathcal{T},\{Z_v\}_{v\in V(\mathcal{T})},\{i_e\}_{e\in E(\mathcal{T})})$$ is said to be included in a tower $\mathbb{T}=\left\{p_n:\Sigma_{n+1}\to\Sigma_n\right\}$ (as in Figure \ref{fig:tree_tower}) if there exist 
\begin{itemize}
\item subsurfaces with geodesic boundary $S_n=\bigsqcup_{v\in V_n(\cT)} S_v$ included in $\Sigma_n$;
\item a family of homeomorphisms $\{h_v:Z_v\to S_v:v\in V(\mathcal{T})\}$;  and
\item a family of embeddings $\{j_e:S_{o(e)}\to S_{t(e)}:e\in E(\mathcal{T})\}$; 
\end{itemize} such that
\begin{itemize}
\item $p_n\circ i_e=\Id$ for every $e\in E_n(\mathcal{T})$;
\item $j_e\circ h_{o(e)}=h_{t(e)}\circ i_e$ for every $e\in E(\mathcal{T})$.
\end{itemize}

For every $n\in\N$ we define the subsurface $S_n^\ast\dans S_n$ by
$$S_n^{\ast}=\bigcup_{e\in E_{n-1}(\mathcal{T})}i_e(S_{o(e)}).$$ 
The (not necessarily connected) surface $S_n$ consists precisely of those surfaces that we want to stabilize (as was illustrated in \S \ref{s.illustration}) whereas the surface $S_n^\ast$ is the isometric lift to the level $n$ of the surfaces we constructed at the level $n-1$.

We will let $X_n$ denote the admissible complement of $S_n$ and $X_n^{\ast}$ the admissible complement of $S_n^{\ast}$. 

\begin{figure}[h!]
\centering
\includegraphics[scale=0.45]{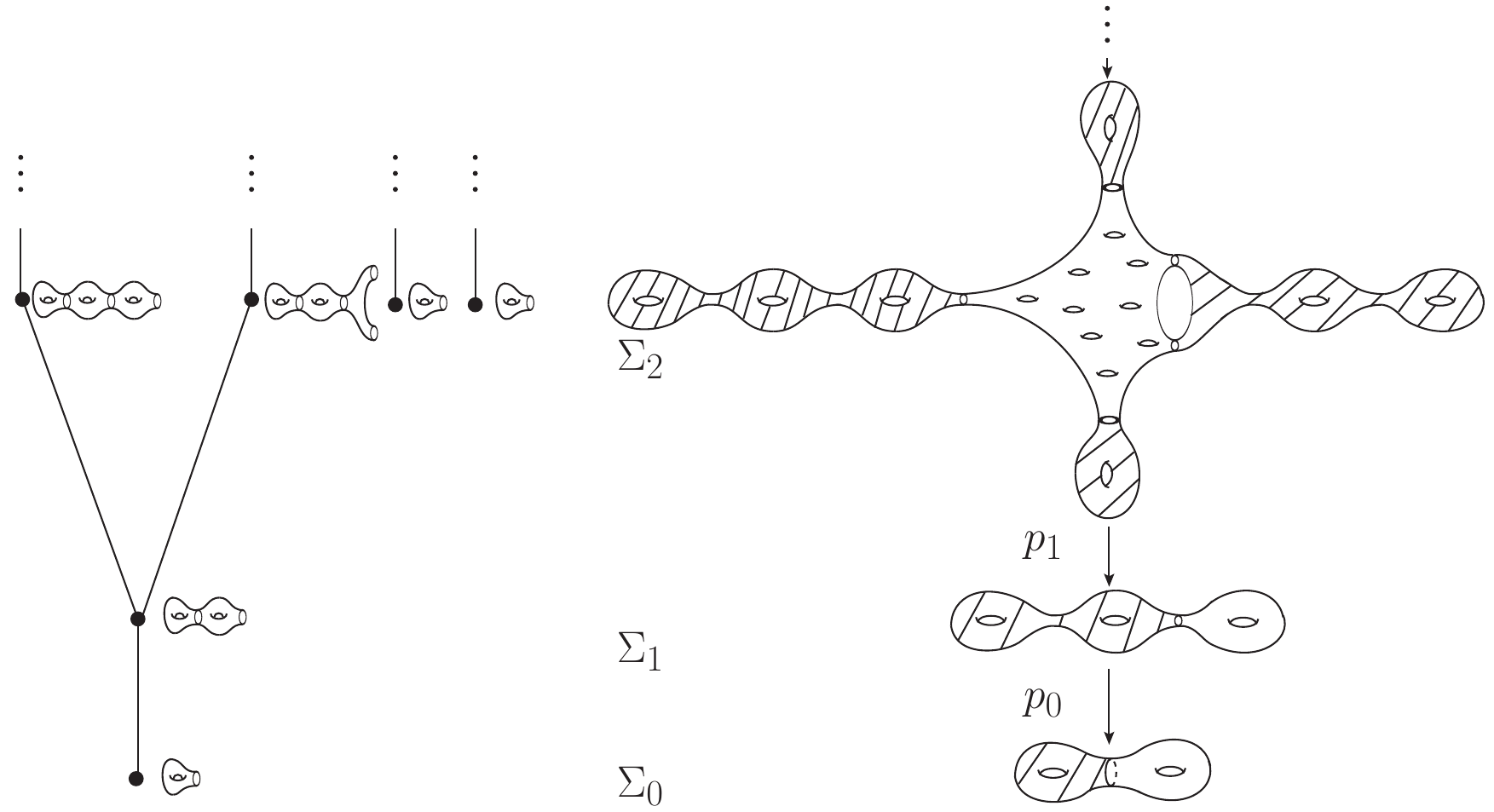}
\caption{Including a forest of surfaces inside a tower of coverings.}\label{fig:tree_tower}
\end{figure}

\paragraph{{\bf Admissible towers --}}  The definition of admissible tower provides the geometric formalization of the intuition explained in \S \ref{s.illustration}.
\begin{definition}[Admissible towers with respect to a forest of surfaces]
\label{def:admissible}
Following the previous notation, we say that \emph{the tower of finite coverings $\mathbb{T}$ is admissible with respect to the  forest of surfaces $\mathcal{S}$} if  $\cS$ is included in $\T$ and if furthermore
\begin{enumerate}
\item the internal systole $\sigma_n$ of $X_n^{\ast}$ tends to infinity with $n$;
\item the boundary of $X_n$ has a half-collar of width $K_n\to\infty$.
\end{enumerate}
When a tower $\mathbb{T}$ is admissible with respect to some surface forest, we call it \emph{an admissible tower}. 
\end{definition}
The following result encapsules the main abstract criteria to control the topology of the leaves of a lamination made by a tower of finite coverings. Let $\mathcal{L}/_{\sim}$ denote the set of leaves of $\mathcal{L}$. 

\begin{theorem}[Topology of the leaves] \label{teo-mainabstract}
Consider a forest of surfaces $$\mathcal{S}=(\mathcal{T},\{S_v\}_{v\in V(\mathcal{T})},\{j_e\}_{e\in E(\mathcal{T})})$$ and $\mathbb{T}=\{p_n:\Sigma_{n+1}\to\Sigma_n\}$ an admissible tower with respect to $\mathcal{S}$ with inverse limit $\mathcal{L}$. Then, the generic leaf of $\mathcal{L}$ is a disk and there exists an injective map $\cE(\mathcal{T})\hookrightarrow \mathcal{L}/_{\sim}$ such that
\begin{itemize}
\item the leaf corresponding to an end $\alpha\in\cE(\mathcal{T})$ is homeomorphic to $\Ss^\alpha$, the open direct limit defined in \ref{ss.surfaces-forests};
\item every leaf which is not included in the image of this map is a disk.
\end{itemize}

\end{theorem}

The rest of the section is devoted to the proof of Theorem \ref{teo-mainabstract}. We first give some geometric properties of admissible towers and deduce that the generic leaf of the solenoid defined by an admissible tower is a disk.

\begin{remark}\label{r.cyl}
Once again we point out that it is possible to adapt this formalism to include cylinder leaves which are not taken into account in the way we have presented forests of surfaces. To do this one should allow degenerate surfaces which involves no new difficulty but makes the presentation more dense. In Section \ref{s.illustration} we already illustrated how cylinders can be embedded and we leave to the reader the adaptations needed to include them in the formalism presented here. 
\end{remark}

\begin{remark}\label{r.forestjustif}
The main difference between trees and forests is that the set of ends of a tree is \emph{compact} and the set of ends of a forest may be \emph{non compact}. For example assume that we want to find a lamination such that all leaves are disks except countably many leaves with prescribed topology (cf Theorem \ref{t.tres}) by including a \emph{tree} of surfaces inside a tower. After Theorem \ref{teo-mainabstract} countably many ends $\alpha_n$ of the tree would provide leaves homeomorphic to the desired surfaces. But this sequence must accumulate to other ends of the tree: undesired surfaces can appear as leaves of the lamination. In order to prove Theorem \ref{t.tres} with our method we will need to use a countable union of trees.
\end{remark}

\subsection{Injectivity radius, decompositions and systoles}\label{ss.decomposition}

Consider a tower of finite coverings $\mathbb{T}=\{p_n:\Sigma_{n+1}\to\Sigma_n\}$ admissible with respect to the forest of surfaces $$\mathcal{S}=(\mathcal{T},\{S_v\}_{v\in V(\mathcal{T})},\{j_e\}_{e\in E(\mathcal{T})}).$$




\paragraph{ {\bf Large injectivity radius --}} We first use the definition of admissibility to prove that there exist points of $X_n$ with arbitrarily large injectivity radii when $n\to\infty$.
This point will yield, in the limit, the leaves which are simply connected.

\begin{lemma}\label{l.large_inj_radius}
There exists $n_0>0$ such that for every $n\geq n_0$ and $x_n\in X_n$ such that $\dist(x_n,\partial X_n)\geq K_n$ we have 
$$\ri(x_n)\geq\min\left(\sigma_n,\frac{K_n}{2}\right).$$
In particular $\ri(x_n)\to\infty$ for such a sequence of points $(x_n)_{n\in\N}$.
\end{lemma}

\begin{proof}
Note first that for every $n\in\N$, $\sys(X_n)\geq\sys(\Sigma_0)>0$ and that $K_n\to\infty$ as $n\to\infty$. Hence there exists $n_0>0$ such that for every $n\geq n_0$,
$$K_n\leq\sys(X_n)\cosh\left(\frac{K_n}{2}\right).$$
Finally $\sigma_n$ is smaller that the internal systole of $X_n$ (which is contained inside $X_n^\ast$). Therefore we can use Lemma \ref{l.inj_radius} and order to prove the result.
\end{proof}

\paragraph{ {\bf Level $i$ subsurfaces --}} A difficulty that we have to deal with in order to prove our main theorems is that there could exist subtrees of the forest $\cT$ with a root appearing at an arbitrarily large floor. These will correspond to components of $S_n$ which are disjoint from the lift of $S_{n-1}$ contained in $\Sigma_n$.


Notice that the surfaces $S_v$ can be written as an exhaustion of subsurfaces corresponding to lifts of surfaces associated to vertices below $v$. We will introduce the subsurfaces $S_{v,i}\dans S_v$ that will denote the (closure of) the difference between the mentioned exhausting subsurfaces. For example, when $i$ is the floor to which $v$ belongs, $S_{v,i}$ will designate the subsurface of $S_v$ which does not come from the embedding of $S_{i-1}$. When $v$ is above floor $i$, $S_{v,i}$ will denote the lift of $S_{w,i}$ inside $S_v$, where $w\in V_i(\cT)$ is the vertex at floor $i$ below $v$.


More precisely, let $V_{n}^{k}(\mathcal{T})$ denote the set of vertices in $V_n(\mathcal{T})$ that belong to a subtree $\mathcal{T}_v$ with $v\in R_k(\mathcal{T})$.

We define the family of \emph{level $i$ surfaces} $\{S_{v,i}:v\in\mathcal{T},i\in\N\}$ as follows. 
\begin{itemize}

\item If $v\in V_{n}^{k}(\mathcal{T})$ and $i<k$ or $i>n$ define $S_{v,i}=\emptyset$.
\item If $v\in V_n^{k}(\mathcal{T})$ and $i=n$ define $S_{v,i}$ as \begin{itemize}
\item $S_v$ if $k=n$ (i.e. if $v$ is a root appearing at floor $k$)
\item $\overline{S_v\setminus j_e(S_{o(e)})}$ with $v=t(e)$ if $n>k$ (this is one of the building blocks defined in \S \ref{ss.directlimit}).
\end{itemize}
\item The family satisfies the recurrence relation $S_{t(e),i}=j_e(S_{o(e),i})$ for $e\in E_m(\mathcal{T})$ and every $m\in\N$ and $i\leq m$.
\end{itemize}
Informally speaking, $S_{v,i}$ is the part of $S_v$ that ``grew'' at level $i$. 

Notice that $$S_v=\bigcup_{i\in\N} S_{v,i}$$  is a decomposition by subsurfaces with geodesic boundary meeting each other along boundary components. We define $S_{n,k}$ as the union of all the subsurfaces $S_{v,i}$ with $v\in V_n(\mathcal{T})$ and $i\leq k$

Then we define $X_{n,k}$ as the admissible complement of $S_{n,k}$. Note that by definition $X_{n,n-1}=X_n^{\ast}$. Note that these surface have the following decomposition
\begin{equation}
\label{eq.deco_Xnk}
X_{n,k}=X_n\cup\bigcup_{v\in V_n(\cT)}\left(\bigcup_{i=k+1}^n S_{v,i}\right).
\end{equation}

\begin{remark}\label{rem_int_boundary_level}
Recall that when $i<n$ we denoted $P_{n,i}=p_i\circ\ldots\circ p_{n-1}:\Sigma_n\to\Sigma_i$. Then for every $v\in V_n(\cT)$ 
\begin{enumerate}
\item the interior of $S_{v,i}$ is mapped inside of $\Int(X_i^\ast)$ by $P_{n,i}$;
\item a boundary component of $S_{v,i}$ is either mapped isometrically onto a boundary component of $X_{i}$ by $P_{n,i}$ or onto a boundary component of $X_{i-1}$ by $P_{n,i-1}$.
\end{enumerate}
\end{remark}

\paragraph{{\bf Increasing internal systoles --}} We will need the following result.

\begin{proposition}
\label{p:systole_internal}
Let $k_n$ be a sequence of integers satisfying $k_n\leq n$ for every $n$ and $\lim_n k_n=\infty$. Then the internal systole of $X_{n,k_n}$ tends to infinity with $n$. Moreover, for every $K>0$, there exists $m\in\N$ so that, if $\gamma_n\subset\partial X_{n,k_n}$ is a sequence of boundary components with $l_{\gamma_n}<K$, then $\gamma_n\subset S_{n,m}$ for every $n\in\N$.
\end{proposition}

\begin{proof} Define $m_n=\textrm{min}_{i> n}\{\sigma_i, K_i\}$ where $\sigma_i$ and $K_i$ are as in the definition of admissible tower. Then $m_n\to\infty$ as $n$ goes to $\infty$.

We consider a sequence $(k_n)_{n\in\N}$ as in the statement of the lemma. We will prove that every closed geodesic of $\Int(X_{n,k_n})$ has a length $\geq m_{k_n}$, which is enough to prove the lemma. For every $n\in\N$, $X_{n,k_n}$ has a decomposition as in \eqref{eq.deco_Xnk}. Recall that this is a decomposition by subsurfaces with geodesic boundary and disjoint interiors. We deduce that there are four possibilities for a 
closed geodesic $\gamma\dans\Int(X_{n,k_n})$.
\begin{enumerate}
\item[\emph{Case 1.}] $\gamma$ is included inside $\Int(X_n)$.
\item[\emph{Case 2.}] $\gamma$ is included inside $\Int(S_{v,i})$ for some $v\in V_n(\cT)$ and $k_n< i\leq n$.
\item[\emph{Case 3.}] $\gamma$ is a boundary component of $S_{v,i}$ for some $v\in V_n(\cT)$ and $k_n+1<i\leq n$.
\item[\emph{Case 4.}] $\gamma$ crosses $\partial S_{v,i}$ for some $v\in V_n(\cT)$ and $k_n+1<i\leq n$.
\end{enumerate}

In Case 1, we automatically have that $l_\gamma\geq\sigma_n\geq m_{k_n}$. 

In Case 2, we use Item 1 of Remark \ref{rem_int_boundary_level} to prove that $P_{n,i}(\gamma)$ is included in $\Int(X_i^\ast)$ so $l_\gamma\geq\sigma_i\geq m_{k_n}$.

In Case 3, we use Item 2 of Remark \ref{rem_int_boundary_level} to get that $P_{n,j}(\gamma)$ is a boundary component of $X_j$ and hence belong to $\Int(X_j^\ast)$ for $j=i$ or $i-1$. This implies that $l_\gamma\geq\sigma_j\geq m_{k_n}$.

And finally in Case 4, we also use Item 2 of Remark \ref{rem_int_boundary_level}. The projection $P_{n,j}(\gamma)$ crosses a boundary component of $X_j$ inside $\Sigma_j$ for $j=i$ or $i-1$. Using Item 2 of Lemma \ref{l.crosscollars} we see that $l_\gamma\geq K_j\geq m_{k_n}$.

For the last part of the proposition, suppose by contradiction that there exist $K>0$, sequences $r_s<n_s$ converging to $+\infty$, and a sequence of boundary components $\gamma_{s}\subset X_{n_s,k_{n_s}}\setminus S_{n_s,r_s}$ satisfying $l_{\gamma_s}<K$ for every $s$. Thus, the internal systole of $X_{n_s,r_s}$ does not converge to $+\infty$ contradicting the first part of the proposition.
\end{proof}

\subsection{Proof of Theorem \ref{teo-mainabstract} }

\paragraph{ {\bf Topology of the generic leaf --}} The first and easiest step in the proof of Theorem \ref{teo-mainabstract} is to prove that the generic leaf of $\cL$ defined by an admissible tower is a disk. Then we will need a further analysis using the forest structure to identify the topology of all leaves.

 A property is said to hold for a generic leaf, if it holds for every leaf in a residual set (that is a countable intersection of dense and open subsets) which is saturated by the lamination.

\begin{lemma} The generic leaf of $\mathcal{L}$ is simply connected. 
\end{lemma}
\begin{proof} For $k\in\N$ define
$$U_k=\{\x\in\mathcal{L}:r_{\textrm{inj}}(x_{m_0})> k \textrm{ for some }m_0\in\N\}.$$ 
First notice that $r_{\textrm{inj}}(x_n)$ is increasing with $n$. Therefore, Proposition \ref{p.Cheeger_Gromov} implies that if $\x\in U_k$ then the injectivity radius of $L_{\x}$ at $\x$ is greater than $k$.

We will show that $U_k$ is open and dense for every $k\in\N$, getting that $\bigcap_{k\in\N}U_k$ is a residual and saturated set all whose leaves are disks. 

\noindent {\it Step 1. $U_k$ is open for every $k\in\N$. } Take $\x\in U_k$ and $m_0\in\N$ so that $r_{\textrm{inj}}(x_{m_0})>k$. Since the injectivity radius function is lower semi-continuous, we can take a neighbourhood $W$ of $x_{m_0}$ in $\Sigma_{m_0}$ such that every point in $W$ has injectivity radius greater than $k$. Then, the set of $\y\in\cL$ satisfying $y_{m_0}\in W$ is an open neighbourhood of $\x$ contained in $U_k$. 

\noindent {\it Step 2. $U_k$ is dense for every $k\in\N$. } Fix two integers $k,m_0\in\N$, as well as a sequence $(x_n)_{n=0,\ldots m_0}$ satisfying $p_n(x_{n+1})=x_n$ when $n<m_0$. We will construct a sequence $\y\in\cL$ such that $y_n=x_n$ when $n\leq m_0$ and $\ri(y_n)>k$ for $n$ large enough (so $\y\in U_k$).

Fix $D\geq\textrm{diam}(\Sigma_n)$ for all $n\leq m_0$. Using Lemma \ref{l.large_inj_radius} we see that there exists $m$ and a point $y_m'\in\Sigma_m$ such that $\ri(y_m')> k+D$. Arguing as in the proof of the minimality of $\cL$ (see Lemma \ref{p.minimality}) we find a sequence  $\y\in\cL$ such that $y_n=x_n$ for $n\leq m_0$ and $\dist(y_m,y_m')\leq D$, which implies that $\ri(y_m)>k$.
\end{proof}

Now we will need to go further and associate a marking to some leaves such that the following dichotomy holds. Unmarked leaves are disks, and the topology of marked leaves is prescribed by the forest.

\paragraph{ {\bf Associated markings -- }} We note $(S_v,q_v)$ the pointed surface associated to the vertex $v$. Recall that inclusions appearing in the forest of surfaces respect the base points, i.e.
$$j_e\left(q_{o(e)}\right)=q_{t(e)}.$$

We can naturally associate a point $\x^\alpha$ in the inverse limit of $\mathbb{T}$ to every end $\alpha$ in $\cE(\mathcal{T})$. For this consider the associated ray $r_\alpha$ (recall Definition \ref{d.forest}) and let $k_0$ denote the floor of $r_\alpha$. The sequence $\x^\alpha=(x_n^\alpha)_{n\in\N}\in\cL$ is defined by 

$$x_n^\alpha=\left\{
\begin{array}{lr}
q_{r_\alpha(n)}                & n\geq k_0\\
P_{k_0,k_0-n}\left(q_{r_\alpha(k_0)}\right)  & n<k_0
\end{array}\right. .$$

\begin{definition}[Markings]\label{d.markouille} We define \emph{the set of markings associated to the admissible tower} as the subset  $(\x^\alpha)_{\alpha\in\cE(\mathcal{T})}$ included in $\mathcal{L}$. The leaves of points $\x^\alpha$ will be called \emph{marked}.
\end{definition}

\begin{lemma}[Different markings give different leaves]
\label{l:dif_comp}
Consider $\alpha$ and $\beta$ different ends of $\cE(\mathcal{T})$. Then
$$\dist(x_n^{\alpha},x_n^{\beta})\To_{n\to\infty}\infty$$
and $L_{\x^\alpha}\neq L_{\x^\beta}$. \end{lemma}

\begin{proof}
If $\alpha\neq\beta$, there exists $n_0$ such that when $n\geq n_0$ the points $x_n^{\alpha}$ and $x_n^{\beta}$ belong to distinct connected components of $S_n$. So any geodesic path between these two points must cross two disjoint half collars of boundary components of $X_n$. Thus, the length of this geodesic path must be greater than $2K_n$. This quantity goes to infinity with $n$ by definition and the lemma is proven.
\end{proof}
 
\paragraph{ {\bf Topology of non-marked leaves --}} A non-marked leaf is by definition the leaf of a sequence $\x\in\cL$ satisfying $\dist(x_n,x_n^\alpha)\to\infty$ for every end of the forest $\alpha\in\cE(\cT)$. We want to prove that such a leaf exists and that it is a disk. We will have to face a difficulty: new roots of the forest can appear at an arbitrary floor and we want to prove that a sequence defining a non-marked leaf goes away from all those roots.

Recall that $V_n^k(\mathcal{T})$ consists of those vertices of $V_n(\mathcal{T})$ that belong to a subtree whose root is at floor $k$. We set
$$Q^k_n=\{q_v:v\in V_n^k(\mathcal{T})\}$$
Note that $Q^k_n\dans S_n\dans \Sigma_n$. Recall that for two integers $n\geq m$, $P_{n,m}$ denotes the projection $p_m\circ\ldots\circ p_{n-1}$.

\begin{lemma}\label{l.dicotomia}
Let $\x\in\mathcal{L}$. We have the following dichotomy.
\begin{itemize}
\item Either there exists $\alpha\in\cE(\mathcal{T})$ such that $\dist(x_n,x_n^{\alpha})$ is uniformly bounded. 
\item Or, for every $k\in\N$ we have
$$\dist(x_n,Q^k_n)\To_{n\to\infty}\infty.$$
\end{itemize}
\end{lemma}

This lemma in particular implies that every leaf is either marked or non-marked.

\begin{proof} Suppose there exist $k\in\N$ and $C>0$ so that $\dist(x_n,Q^k_n)\leq C$ for every $n\in\N$. Then, there exists a sequence of points $q_n\in Q^k_n$ such that for every $n\in\N$, $\dist(x_n,q_n)\leq C$.

Fix $m\geq k$ and define for $n\geq m$ the sequence $q_n^m=P_{n,m}(q_n)\in\Sigma_m$. For such a pair $(m,n)$ we have by definition $q_n^m\in Q_m^k$. The set $Q_m^k$ is finite so for a given $m\geq k$ infinitely many of the points $q_n^m$ coincide. Hence a diagonal argument provides an infinite subsequence of integers $(n_i)_{i\in\N}$ such that for every $m\geq k$ the sequence $(q^m_{n_i})_{i\in\N}$ of points of $Q^k_m$  is eventually constant (we denote $y_m$ the common value), and satisfies $p_m(q^{m+1}_{n_i})=q^{m}_{n_i}$ and $\dist(x_m,q^m_{n_i})\leq C$ for $i$ large enough.

Hence the sequence $(y_m)_{m\geq k}$ is the tail of a point of $\cL$ which must be marked by some end $\alpha\in\cE(\cT)$ (this is because for every $m\geq k$, $y_m$ is the marked point $q_v$ of some surface $S_v$). This implies that $y_m=x_m^{\alpha}$ for every $m$ and finally that $\dist(x_m,x_m^\alpha)\leq C$ and the lemma follows.
\end{proof}

\begin{lemma}[Controlling the topology of non-marked leaves]
\label{l:discs}
Let $\x=(x_n)_{n\in\N}\in\cL$ such that for every $\alpha\in\cE(\mathcal{T})$ we have
$$\dist(x_n,x_n^{\alpha})\To_{n\to\infty}\infty.$$
Then we have
$$\ri(x_n)\to\infty.$$
In particular $L_{\x}$ is a disk by Proposition \ref{p.Cheeger_Gromov}. 
\end{lemma}

\begin{proof}
We must prove that the length of every geodesic loop based at $x_n$ tends to infinity with $n$. Arguing by contradiction, suppose there exists a sequence of geodesic loops $\gamma_{n}$ based at $x_{n}$ with uniformly bounded lengths $l_{\gamma_{n}}$. 

Applying the dichotomy of Lemma \ref{l.dicotomia} and the fact that $l_{\gamma_n}$ is uniformly bounded, we get $\dist(\gamma_n,Q^k_n)\to\infty$ for every $k\in\N$. We claim that for every $k\in\N$ there exists $m_k\in\N$ so that $\gamma_n\dans X_{n,k}$ for every $n\geq m_k$. To see this, first notice that the components of $S_{n,k}$ consist of $(1:1)$-lifts of components of $S_k$, and therefore have uniformly bounded diameter. On the other hand, each component of $S_{n,k}$ contains a point $q_v\in Q^k_n$. Since $\dist(\gamma_n,Q^k_n)\to\infty$ we conclude that $\gamma_n$ does not meet $S_{n,k}$ for $n$ large enough as desired. Then, define $k_n:=\mathrm{max}\{i<n:m_i<n\}$. Clearly, $k_n\to+\infty$ and, by definition of $m_i$, the curve $\gamma_n$ is included in $X_{n,k_n}$ for every $n\in\N$. 

There are two possibilities.

\noindent \emph{Case 1. $\gamma_n$ is not isotopic to a boundary component of $X_{n,k_n}$.} Applying Proposition \ref{p:systole_internal} we get that the internal systole of $X_{n,k_n}$ goes to $+\infty$. On the other hand, since $\gamma_n$ is not isotopic to a boundary component of $X_{n,k_n}$, we have that $l_{\gamma_n}$ is greater than its internal systole. Therefore, Case 1 happens for finitely many $n$.

\noindent \emph{Case 2. $\gamma_n$ is isotopic to a boundary component of $X_{n,k_n}$.} Denote by $\beta_n$ the boundary component of $X_{n,k_n}$ isotopic to $\gamma_n$. We have $l_{\gamma_n}\geq l_{\beta_n}$ so $\beta_n$ has bounded length. Applying the second part of Proposition \ref{p:systole_internal} we get $m\in\N$ so that $\beta_n\dans S_{n,m}$ for every $n$. In particular, since $S_{n,m}$ has uniformly bounded diameter and contains a point in $Q^k_n$, we get that the distance $D_n$ from $\gamma_n$ to $\beta_n$ tends to infinity. Finally, arguing as in the proof of Lemma \ref{l.inj_radius} we obtain a lower bound.

$$l_{\gamma_n}\geq \cosh(D_n) l_{\beta_n}\geq  \cosh(D_n)\sys(\Sigma_0).$$
 This contradicts that $l_{\gamma_n}$ is uniformly bounded.


\end{proof}

We will now end the proof of Theorem \ref{teo-mainabstract} and characterize the topology of marked leaves. 

\paragraph{ {\bf Embedding direct limits --}} Given  $\alpha$, an end of $\mathcal{T}$ represented by a sequence  $(e_n)_n$ of edges, we denote respectively by $\Ss^\alpha$ and $\overline{\Ss}^\alpha$ the open and geometric direct limits of the sequence of isometric embeddings $\{j_{e_n}:S_{o(e_n)}\to S_{t(e_n)}\}$ (recall definition \ref{d.geom_direct_limit}). Recall that $\Ss^\alpha$ is diffeomorphic to the interior of $\overline{\Ss}^\alpha$.
\begin{lemma}
\label{l.encaje_Lalpha}
For every end $\alpha\in\cE(\mathcal{T})$ there exists an isometric embedding $\phi:\overline{\Ss}^\alpha\to L_{\x^\alpha}$.
\end{lemma}

\begin{proof} Set $\alpha=(e_n)_n$ and $v_n=r_{\alpha}(n)$ for every $n$ greater or equal than the floor where $r_{\alpha}$ starts. By Proposition \ref{p.Cheeger_Gromov} the pointed leaf $(L_{\x^\alpha},\x^\alpha)$ is the Cheeger-Gromov limit of the sequence of pointed surfaces $(\Sigma_n,x^{\alpha}_n)_{n\in\N}$. More precisely, the proof of that Proposition shows that for every compact domain $D\dans L_{\x^\alpha}$ there exists $n_0$ such that for every $n\geq n_0$, the projection on the $n$-th coordinate induces an isometric embedding 
$$\Pi_n:(D,\x^\alpha)\to(\Sigma_n,x^{\alpha}_n).$$

For $n<m$, set $j_{m,n}=j_{e_{m-1}}\circ\ldots\circ j_{e_n}:S_{v_n}\to S_{v_m}$. This is an isometric embedding whose image is included inside the $R_n$-neighbourhood of $x_m^{\alpha}$ for some $R_n$ independent of $m>n$. Using the property stated above, for $m$ large enough the inverse of $\Pi_m$ induces a isometric embedding $\phi_n:S_{v_n}\to L_{\x^\alpha}$. This yields a sequence of isometric embeddings which satisfy $\phi_{n+1}\circ j_{e_n}=\phi_n$ (note that we have $j_{m,n}\circ\Pi_n=\Pi_m$  when the left-hand term is defined).

Using the universal property of direct limits and the fact that the leaf $L_{\x^\alpha}$ is a complete Riemannian surface, we see that there exists an isometric embedding of the geometric direct limit $\phi:\overline{\Ss}^\alpha\to L_{\x^\alpha}$.
\end{proof}

\paragraph{ {\bf Topology of the marked leaf --}} The embedding obtained in Lemma \ref{l.encaje_Lalpha} might not be surjective. In fact, its image can be complicated from the geometric point of view since we don't control the lengths of the boundary components of the surfaces $S_{v_n}$. Nevertheless, using Proposition \ref{p.recognize_surface} and the geometric properties of admissible towers, we will prove that this embedding contains all the topological information of the complete hyperbolic surface $L_{\x^\alpha}$.

Recall the definition of open direct limit in \S \ref{ss.directlimit}. Given an end $\alpha=(e_n)$ we define $\Ss^{\alpha}$ as the open direct limit of the sequence of embeddings $(j_{e_n})$.
\begin{proposition}
\label{p:topology_leaves}
The leaf of a sequence $\x^{\alpha}$ is diffeomorphic to the open direct limit $\Ss^\alpha$.
\end{proposition}

This proposition finishes the topological characterization of all leaves of $\cL$ and the proof of Theorem \ref{teo-mainabstract}.

We will fix $\alpha\in\cE(\mathcal{T})$ and note $v_n=r_{\alpha}(n)$. Consider the closed surface defined in Lemma \ref{l.encaje_Lalpha}
$$S=\phi(\overline{\Ss}^\alpha).$$
This is a closed surface with (possibly) geodesic boundary (see Proposition \ref{p:geom_limit}).

We shall prove Proposition \ref{p:topology_leaves} in two steps, by checking that the pair $(L_{\x^\alpha},S)$ satisfies the hypotheses of Proposition \ref{p.recognize_surface}. For this, we need two lemmas.

\begin{lemma}[No closed geodesic outside of $S$]\label{l.no_closed_geod_outside} Let $C$ be a connected component of $L_{\x^\alpha}\moins S$. Then $C$ contains no closed geodesic.
\end{lemma}
\begin{proof} Suppose there exists a closed geodesic $\gamma$ included in $C$. The projections $\gamma_n=\Pi_n(\gamma)$ define a sequence of closed geodesics in $\Sigma_n$
 disjoint from $S_{v_n}$ with the same length and located at a uniform distance to $x_n^{\alpha}$ (by definition of the leaf $L_{\x^\alpha}$). In particular
  it must be disjoint from all other subsurfaces $S_{v}$ with $v\in V_n(\mathcal{T})$, when $n$ is large enough (by Lemma
   \ref{l:dif_comp}). Hence it must be completely included in $\Int(X_n)$ for $n$ large enough, contradicting that the internal systole of $X_n$ goes to infinity as $n$ grows.
\end{proof}

\begin{lemma}[No connection of boundary components outside of $S$]\label{l.no_connexion}
Let $C$ be a connected component of $L_{\x^\alpha}\moins S$. Then the boundary $\partial C$ is connected.
\end{lemma}

\begin{proof}
Assume that $\partial C$ has more than one boundary component (that can be a closed geodesic or a complete geodesic). Then there exists a simple geodesic arc $\gamma$ contained in $\overline{C}$ and connecting two points $y$ and $z$ of these two boundary components. Consider $D_1,D_2$, two disks inside $L_{\x^\alpha}$ centered at $y$ and $z$ respectively. 

Now note that $S$ is an increasing union of compact surfaces with boundary $S_n$ such that $\Pi_n:S_n\to S_{v_n}$ is an isometry for every $n$. For $n$ large enough, there exist two points $y_n\in D_1\cap\partial S_n$ and $z_n\in D_2\cap \partial S_n$ and a simple geodesic arc $\gamma_n$ between them, that is outside $\textrm{Int}(S_n)$ and whose length is bounded independently of $n$. As a consequence there exists a compact domain $D$ containing all geodesics $\gamma_n$. For $n$ large enough the projection $\Pi_n|_D:D\to\Sigma_n$ is an isometric embedding and the projection of the $\gamma_n$ (still denoted by $\gamma_n$) to $\Sigma_n$ is a path in $\Sigma_n\moins S_{v_n}$ connecting two points, denoted abusively by $y_n,z_n$, of $\partial S_{v_n}$. Fix such an $n$ and assume that the quantity $K_n$ (recall that it denotes a lower bounds of the width of half-collars of boundary components of $X_n$) is $>l_{\gamma_n}$. There are two possibilities.

If $y_n$ and $z_n$ belong to two distincts connected components of $\partial S_{v_n}$ then, arguing as in Lemma \ref{l:dif_comp} (two large and disjoint half collars are attached to these components) we see that $l_{\gamma_n}\geq 2K_n$ which is a contradiction.

If $y_n$ and $z_n$ belong to the same boundary component, called $\alpha_n$, then $\gamma_n$, which has length $<K_n$, must be completely included inside a collar about $\alpha_n$. This implies that the geodesics $\alpha_n$ and the simple geodesic arc $\gamma_n$ form a bigon, which is absurd. 
\end{proof}

\section{Including forests of surfaces in towers}\label{s.including_forest}
The rest of the paper is devoted to the proof of Theorem  \ref{t.uno} and Theorem \ref{t.tres}. Both theorems will be deduced from the more general result

\begin{proposition} \label{p.forestinclusion}Given a forest of surfaces $\mathcal{S}$ there exists a tower of finite coverings $\mathbb{T}$ which is admissible with respect to $\mathcal{S}$.
\end{proposition}
\begin{proof} Consider a forest of surfaces $$\mathcal{S}=(\mathcal{T},\{Z_v\}_{v\in V(\mathcal{T})},\{i_e\}_{e\in E(\mathcal{T})}).$$
Proceeding inductively, we will construct a tower of finite coverings $\mathbb{T}$ which is admissible with respect to $\mathcal{S}$.

\paragraph{ {\bf The base case --}} Consider a hyperbolic surface $\Sigma_0$ containing a set of pairwise disjoint subsurfaces with geodesic boundary $\{S_v:v\in V_0(\mathcal{T})\}$, so that each $S_v$ is homeomorphic to $Z_v$. Define $S_0:=\bigcup_{v\in V_0(\mathcal{T})} S_v$ and $X_0$ as it admissible complement.

\paragraph{ {\bf The induction hypothesis --}} Suppose we have already included $\mathcal{S}$ up to the floor $k$. Namely, we have:
\begin{itemize}
\item Finite coverings $p_i:\Sigma_{i+1}\to\Sigma_i$ for $i=0,\ldots,k-1$

\item Subsurfaces $S_i\dans \Sigma_i$ with $S_i=\bigsqcup_{v\in V_i(\mathcal{T})} S_v$ where each $S_v$ is homeomorphic to $Z_v$ and $i=1,\ldots,k$
\item A family of lifts $\{j_e:S_{o(e)}\to S_{t(e)}\}$ where $e$ ranges over $E_i(\mathcal{T})$ with $i\leq k-1$

\item A family of homeomorphisms $\{h_v:Z_v\to S_v;\,v\in V_i(\mathcal{T}),i\leq k\}$

such that for every $e\in E_{i}(\mathcal{T})$ we have $j_e\circ h_{o(e)}=h_{t(e)}\circ i_e$

\item The internal systole of $X_i^{\ast}$ and half-collar width of the boundary components of $X_i$ are greater than $i$ for $i=1,\ldots,k$; where $X_i^{\ast}$ is the admissible complement of $$S_{i}^{\ast}=\bigcup_{e\in E_{i-1}}j_e(S_{o(e)})$$ and $X_i$ is the admissible complement of $S_i$.

\end{itemize}
\paragraph{ {\bf The induction step --}}  We will now use the tools described in Section \ref{s.toolbox} in order to construct the desired covering map $p_k:\Sigma_{k+1}\to\Sigma_k$.\\

\noindent {\it Step 1. Creating new roots and space for extending level $k$ surfaces.} Take $g\in\N$ such that every surface in 
$\{Z_{t(e)};\,e\in E_{k}(\mathcal{T})\}$ can be realized a subsurface with geodesic boundary
 of a hyperbolic surface with geodesic boundary with  genus $g$ and two boundary components. 
Define $$S_0=\bigsqcup_{v\in R_{k+1}(\mathcal{T})}Z_v$$ 
the union of surfaces corresponding to the roots of $\cT$ appearing at floor $k+1$.

Applying Theorem \ref{t.finitecovering} and Remark \ref{r.room} we can construct a finite covering $q_1:\Sigma^{(1)}\to \Sigma_{k}$ such that $\Sigma^{(1)}$ admits an admissible decomposition $(X^\ast,S^\ast)$ satisfying
\begin{itemize} 
\item $S^\ast$ decomposes as 
$$S^{\ast}=\bigsqcup_{e\in E_k(\mathcal{T})} S^{\ast}_{t(e)}$$
such that $q_1$ restricted to $S^{\ast}_{t(e)}$ is an isometry onto $S_{o(e)}$ for every $e\in E_{k}(\mathcal{T})$. Denote by $j^{(1)}_e:S_{o(e)}\to S^{\ast}_{t(e)}$ the inverses of these restrictions;
\item the internal systole of $X^\ast$ is $\geq k+1$;
\item $X^{\ast}$ contains a subsurface with geodesic boundary $X=X_1\sqcup X_2$ such that:
\begin{itemize}
\item $X_1$ is homeomorphic to $S_0$;
\item For every boundary component $\alpha_j$ of $X^{\ast}$ there exists a genus $g$ surface $A_j$ included in $X_2$ with two boundary components one of which is $\alpha_j$. Moreover, subsurfaces corresponding to different boundary components are disjoint.
\end{itemize}
\end{itemize}
Note that by this last condition, $X_1\dans\Int(X^\ast)$.\\

\noindent {\it Step 2. Recognizing level $k+1$ subsurfaces --}  We shall now proceed to construct both families $S^{(1)}=\{S^{(1)}_v:v\in V_{k+1}(\mathcal{T})\}$ and $h^{(1)}=\{h_v^{(1)}:v\in V_{k+1}(\mathcal{T})\}$ simultaneously.

\begin{itemize}
\item {\it Case $v\in R_{k+1}(\mathcal{T})$. } 

By construction, we have a decomposition $X_1=\bigsqcup_{v\in R_{k+1}(\mathcal{T})}S^{(1)}_{v}$ where each $S^{(1)}_{v}$ is homeomorphic to $Z_v$. For every $v\in R_{k+1}(\mathcal{T})$ define $h_{v}$ as any homeomorphism between $S^{(1)}_{v}$ and $Z_v$.

\item {\it Case $v=t(e)$ for some $e\in E_k(\mathcal{T})$.}

Set $A=\Int(X_2)\dans\Sigma^{(1)}$. This is a key point of the whole construction. Recall that by definition of good inclusion we have that $\overline{Z_{t(e)}\setminus i_e(Z_{o(e)})}$ admits a hyperbolic structure with geodesic boundary. Therefore, the choice of $g$ and the construction of $A$, imply that there is enough room to construct a pairwise disjoint family of subsurfaces with geodesic boundary $\{S^{(1)}_{t(e)};\,e\in E_{k}(\mathcal{T})\}$ and a family of homeomorphisms $\{h^{(1)}_{t(e)}:Z_{t(e)}\to S^{(1)}_{t(e)};e\in E_{k}(\mathcal{T})\}$ so that
\begin{enumerate}
\item $S^{(1)}_{t(e)}\subset S^{\ast}_{t(e)}\cup A$ 
\item $h^{(1)}_{t(e)}\circ i_e=j^{(1)}_e\circ h_{o(e)}$
\end{enumerate}

\end{itemize}

Define $$S^{(1)}=\bigcup_{v\in V_{k+1}(\mathcal{T})}S^{(1)}_v.$$ 
Using notations coherent with \S \ref{ss.decomposition}, we define $S^{(1)}_{v,k+1}$ as follows:
\begin{itemize}
\item if $v\in R_{k+1}(\mathcal{T})$, define $$S^{(1)}_{v,k+1}=S^{(1)}_{v}$$
\item Otherwise $v=t(e)$ for some $e\in E_k(\mathcal{T})$ and we define $$S^{(1)}_{v,k+1}=\overline{S^{(1)}_{v}\setminus S^{\ast}_{t(e)}}$$
\end{itemize}

Notice that, since the internal systole of $X^{\ast}$ is greater than $k+1$, so are those of surfaces $S^{(1)}_{v,k+1}$. Also for the same reason, the boundary components of these surfaces which are not in the interior of the $S_v^{(1)}$ have length greater than $k+1$. \\

\noindent {\it Step 3. Increasing collars of the admissible complement --} Applying Theorem \ref{t.finitecovering} again, we can construct a covering $q_2:\Sigma_{k+1}\to \Sigma^{(1)}$ with a subsurface $S_{k+1}$ such that
\begin{itemize}
\item $S_{k+1}$ decomposes as $$S_{k+1}=\bigsqcup_{v\in V_{k+1}(\mathcal{T})}S_v$$ where $q$ restricted to $S_v$ is an isometry onto $S^{(1)}_v$ for every $v\in V_{k+1}(\mathcal{T})$; 
\item The admissible complement of $S_{k+1}$, denoted by $X_{k+1}$, satisfies:
\begin{enumerate}
\item The internal systole of $X_{k+1}$ is greater than $k+1$
\item The boundary of $X_{k+1}$ has a collar of width $k+1$
\end{enumerate}

\end{itemize}

Let $\phi_v^{(2)}$ denote the inverse of $q_2|_{S_v}$ for every $v\in V_{k+1}(\mathcal{T})$. We define
\begin{itemize}
\item  $p_k$ as the composition $q_2\circ q_1$; 
\item $j_e$ as the composition $\phi_{t(e)}^{(2)}\circ j_e^{(1)}$ for every $e\in E_{k}(\mathcal{T})$;
\item $h_v$ as the composition $\phi_v^{(2)}\circ h^{(1)}_v$ for every $v\in V_{k+1}(\mathcal{T})$;
\item $S_{v,k+1}$ as the images by the maps $\phi_v^{(2)}$ of the surfaces $S^{(1)}_{v,k+1}$;
\end{itemize}
We define $S_{k+1}^{\ast}=\bigcup_{e\in E_{i-1}}j_e(S_{o(e)})$ and $X_{k+1}^{\ast}$ as its admissible complement.\\

\noindent {\it Step 4. The internal systole of $X_{k+1}^{\ast}$ --} It remains to check that the internal systole of $X^{\ast}_{k+1}$ is greater than $k+1$. For this, we consider the decomposition stated in \eqref{eq.deco_Xnk}
$$X_{k+1}^\ast=X_{k+1}\cup\bigcup_{v\in V_{k+1}}S_{v,k+1}.$$

%

The following properties are satisfied.
\begin{itemize}
\item The boundary components of subsurfaces in the decomposition lying on the interior of $X_{k+1}^{\ast}$ satisfy
\begin{enumerate}
\item They are contained in the boundary of $X_{k+1}$ and therefore have a half-collar width greater than $k+1$ by construction. 
\item They belong to a boundary component of $S_{v,k+1}$ not contained in the interior of the $S_v$, for some $v\in V_{k+1}(\mathcal{T})$. Therefore they have length greater than $k+1$ (see the end of Step 2)
\end{enumerate}
\item The internal systole of $X_{k+1}$ is greater than $k+1$
\item The internal systole of each surface $S_{v,k+1}$ is greater than $k+1$

\end{itemize}

Now we are in condition to apply Lemma \ref{l.systole-gluing} to the decomposition of $X_{k+1}^{\ast}$ and get that the internal systole of $X_{k+1}^{\ast}$ is greater than $k+1$ as desired. This finishes the proof of the Proposition.
\end{proof}

\section{Proof of main theorems}\label{s.final_proof}
\subsection{Combinatorial representation of surfaces}
In order to prove Theorems \ref{t.uno} and \ref{t.tres} we will apply Proposition \ref{p.forestinclusion} and Theorem \ref{teo-mainabstract} to particular choices of surface forests. We proceed to define coding trees which will give us a combinatorial framework to construct surface forests. We recall here that we will prove this theorems ignoring the cylinder leafs, but that the arguments may be easily addapted to include cylinders as leaves (c.f. Remark \ref{r.cyl}). 

\paragraph{ {\bf Coding trees --}} In \cite{BWal}, Bavard and Walker define a combinatorial object (they call it \emph{core tree}), which is a rooted and colored tree that gives a combinatorial framework for describing open surfaces. We give a variation of their setting that is more suited to our purposes. 

\begin{definition} A \emph{coding tree} is a connected and rooted tree $\Lambda$ having two types of vertices
\begin{itemize}
\item \emph{boundary vertices} that must have valency $1$ or $2$
\item \emph{simple vertices} that must have valency $1$, $2$ or $3$

\end{itemize}
and satisfying the following properties: 
\begin{itemize}
\item The \emph{root} must be a simple vertex.
\item All simple vertices must have valency strictly greater than one, with the only possible exception of the root. 
\item  Edges must join a boundary vertex with a simple vertex (they cannot joint two vertices of the same type).
\end{itemize}
\end{definition}

Any coding tree $\Lambda$ can be written as the union of the balls of radius one around its simple vertices. Notice that this balls meet at boundary vertices, and consist of simple vertices with one, two or three adjacent vertices, that will be calles valency one, two or three basic pieces respectively. Call these subgraphs the \emph{basic subgraphs} of $\Lambda$.

\paragraph{ {\bf From coding trees to pointed surfaces --}} We proceed to describe how to associate a pointed surface to coding tree $\Lambda$. First, associate to each basic subgraph of $\Lambda$ a surface with boundary as follows
\begin{itemize}
\item Valency one simple vertex $\leftrightarrow$ surface of genus one with one  boundary component
\item Valency two simple vertex $\leftrightarrow$ surface of genus one with two boundary components
\item Valency three simple vertex $\leftrightarrow$ surface of genus zero with three boundary components
\end{itemize} 

\begin{figure}[h!]
\centering
\includegraphics[scale=0.4]{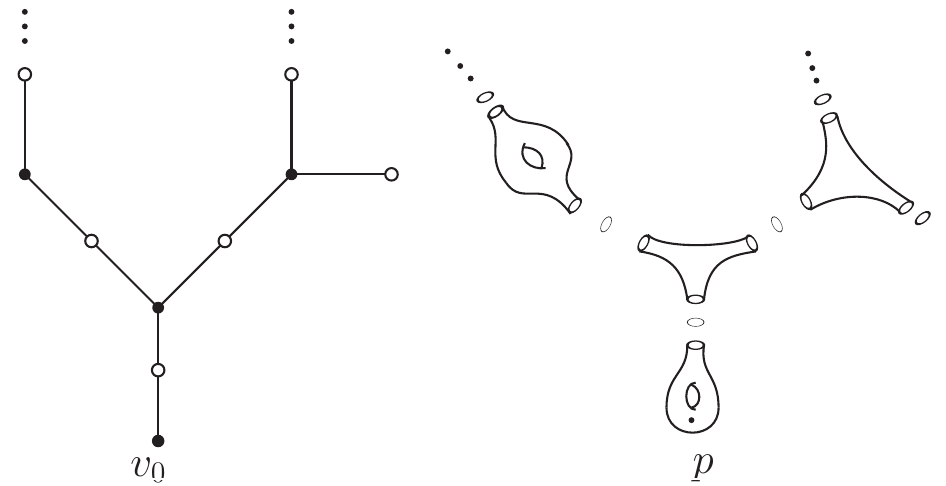}
\caption{From coding trees to surfaces.}\label{fig:tree_surf}
\end{figure}

Consider also a correspondence between the boundary vertices of each basic piece and the boundary components of its associated surface. Then, glue the boundary components of two different surfaces if their associated boundary vertices are equal. To obtain a pointed surface, take any point in the interior of the subsurface corresponding to the root of $\Lambda$. See Figure \ref{fig:tree_surf}.

Notice that the topological type of the resulting pointed surface does not depend on the choices that we have made. Abusing notation we will denote $\Sigma_{\Lambda}$ to either the topological type obtained by the previous construction or to a particular representative of this equivalence class of topological surfaces.

\paragraph{ {\bf Every surface can be obtained from a coding tree --}}

The proof of the following Lemma is a straightforward adaptation of \cite[Lemma 2.3.1]{BWal}  which relies on the classification of open surfaces given by classifying triples. 

\begin{lemma}
\label{t:bavad_walker}
Let $S$ be an open orientable surface other than the disk and the annulus. Then there exists a coding tree $\Lambda$ such that $S$ is homeomorphic  $\Sigma_{\Lambda}$.
\end{lemma}

\paragraph{ {\bf Good inclusions between coding trees --}} Consider coding trees $\Lambda_1$ and $\Lambda_2$. We say that an inclusion $j:\Lambda_1\to\Lambda_2$ is a \emph{good inclusion of coding trees} if $j$ is an injective map that preserves the graph structure, the roots, the vertices types, and whose image boundary consists of a union of boundary vertices.

Notice that a good inclusion of coding trees $j:\Lambda_1\to\Lambda_2$ naturally induces a good inclusion of surfaces $j_{\Sigma}:\Sigma_{\Lambda_1}\to\Sigma_{\Lambda_2}$ that respects subsurfaces coming from basic subgraphs. 

When there is a bijective good inclusion between two coding trees we will say they are isomorphic. Given a coding tree $\Lambda$ we will note $[\Lambda]$ its class of isomorphism.

\paragraph{ {\bf Forests of coding trees --}}

In the rest of the section we want to show how to use this combinatorial description for surfaces in order to exhibit a combinatorial way to organize non-compact surfaces. 

Analogously to the definition of forest of surfaces, we define a \emph{forest of coding trees} as a triple $$\mathcal{T}^{\ast}=(\mathcal{T},\{\Lambda_v\}_{v\in V(\mathcal{T})},\{j_e\}_{e\in E(\mathcal{T})})$$ where $\{\Lambda_v\}_{v\in V(\mathcal{T})}$ is a family of finite coding trees and $\{j_e\}_{e\in E(\mathcal{T})}$ is a family of good inclusions $j_e:\Lambda_{o(e)}\to\Lambda_{t(e)}$. 

There is also an analogous definition of set of \emph{limit coding trees} that associates a coding tree $\Lambda^{\alpha}$ to each end $\alpha\in\cE(\mathcal{T})$.

 \paragraph{{\bf From forests of coding trees to forests of surfaces --}}
Consider $\mathcal{T}^{\ast}$ a forest of coding trees over $\mathcal{T}$. We define its associated forest of surfaces as follows (see Figure \ref{fig:forestouille}):

\begin{figure}[h!]
\centering
\includegraphics[scale=0.6]{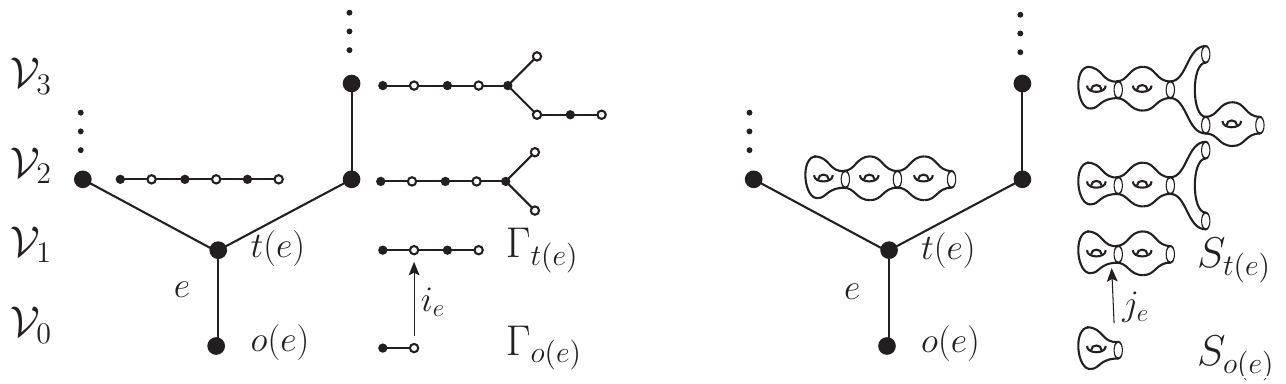}
\caption{A forest of coding trees and its associated forest of surfaces.}\label{fig:forestouille}
\end{figure}

For every vertex $v\in V(\mathcal{T})$ take a surface $S_v$ with the topological type of $\Sigma_{\Lambda_v}$. Then, for every edge $e\in E(\mathcal{T})$ consider a good inclusion of surfaces $({j_e})_{\Sigma}:S_{o(e)}\to S_{t(e)}$ preserving the subsurfaces corresponding to basic subgraphs. We will note $\Sigma_{\mathcal{T}^{\ast}}$ the resulting forest of surfaces. 

Although several choices were made in the previous construction, different choices give rise to \emph{isomorphic trees of surfaces}. Namely, if

$$\Sigma^{(1)}_{\mathcal{T}^{\ast}}=(\mathcal{T},\{S_v\}_{v\in V(\mathcal{T})},\{{(j_e)}_{\Sigma}\}_{e\in E(\mathcal{T})})$$ and $$\Sigma^{(2)}_{\mathcal{T}^{\ast}}=(\mathcal{T},\{Z_v\}_{v\in V(\mathcal{T})},\{{(i_e)}_{\Sigma}\}_{e\in E(\mathcal{T})})$$
are trees of surfaces obtained taking different choices, we can find a family of homeomorphisms $\{h_v:Z_v\to S_v\}_{v\in V(\mathcal{T})}$ preserving roots and satisfying $(j_e)_\Sigma\circ h_{o(e)}=h_{t(e)}\circ (i_e)_\Sigma$ for every $e\in E(\mathcal{T})$. In particular, if $\alpha$ is an end of $\cE(\mathcal{T})$, the corresponding limit surfaces $\mathbf S^{\alpha}$ and $\mathbf Z^{\alpha}$ are homeomorphic. 

\begin{remark} \label{r.limitcorrespondence} It follows directly from the definitions that $\mathbf S^{\alpha}$ is homeomorphic to the interior of $\Sigma_{\Lambda^{\alpha}}$

\end{remark}

\subsection{Constructing surface forests}\label{s.constructingcodingtrees} By Theorem \ref{teo-mainabstract} and Proposition \ref{p.forestinclusion}, in order to prove Theorem \ref{t.uno}, it is enough to show that the family of all surfaces can be realized \emph{inside} the set of ends  of a forest of surfaces. In order to prove Theorem \ref{t.tres}, it is enough to show that every \emph{finite or countable} family of open surfaces can be realized  \emph{as} the set of ends of  a forest of surfaces (see Remark \ref{r.forestjustif}).

On the other hand, by Remark \ref{r.limitcorrespondence} and Lemma \ref{t:bavad_walker} this reduces to realizing certain sets of coding trees as the set of limit coding trees of a forest.

\paragraph{ {\bf The universal forest of coding trees--} }
In the particular case of Theorem \ref{t.uno}, it is enough to construct a forest of coding trees $\mathcal{T}^{\ast}$ so that every coding tree appears as a limit coding tree.

Consider $(\Lambda,v_0)$ a coding tree. Notice that even though not every ball included in $\Lambda$ is a coding tree, if $v_0$ is the root, then $B_{\Lambda}(v_0,2n+1)$ is also a coding tree for every $n\in\N$. 

\paragraph{ {\bf The construction of $\mathcal{T}^{\ast}$ --}} We start defining the underlying forest $\mathcal{T}$, for this we define $$V_n(\mathcal{T})=\Bigg\{\bigg[B_{\Lambda}(v_0,2n+1)\bigg]:(\Lambda,v_0) \text{ coding tree}\Bigg\}$$ 
 
Since vertices in coding trees have bounded valency, $V_n(\mathcal{T})$ is a finite set. We define that $$([\Lambda_1],[\Lambda_2])\in E_n(\mathcal{T})\subset V_n(\mathcal{T})\times V_{n+1}(\mathcal{T})$$ if there exists a good inclusion $i:\Lambda_1\to\Lambda_2$ (recall that good inclusions preserve the root). Given $[\Omega]\in V(\mathcal{T})$ define $\Lambda_{[\Omega]}$ as any representative of $[\Omega]$, and given 
$$e=([\Omega_1],[\Omega_2])\in E(\mathcal{T})$$ 
define $j_e$ as any good inclusion from $\Lambda_{[\Omega_1]}$ to $\Lambda_{[\Omega_2]}$.  Summarizing, our forest of coding trees  is $$\mathcal{T}^{\ast}=(\mathcal{T},\{\Lambda_{[\Omega]}\}_{[\Omega]\in V(\mathcal{T})},{\{ j_e\}}_{e\in E(\mathcal{T})})$$

Finally, notice that if $(\Omega,v_0)$ is a coding tree, the ray $r(n)=[B_{\Omega}(v_0,2n+1)]\subset\mathcal{T}$ represents an end $\alpha\in\cE(\mathcal{T})$ satisfying $\Lambda^{\alpha}\cong (\Omega,v_0)$.

\paragraph{ {\bf The countable forest of coding trees --}} In order to prove Theorem \ref{t.tres}, we have to construct a forest of coding trees $\mathcal{T}^{\ast}$ that realizes any given countable family of coding trees $\{\Lambda_1,\ldots,\Lambda_i,\ldots\}$ as its set of limit coding trees. 

\paragraph{ {\bf The construction of $\mathcal{T}^{\ast}$:}}  We start defining the underlying forest. For this, set $$V_n(\mathcal{T})=\{B_{\Lambda_i}(v_i,r_{n,i}):i=1,\ldots,n\}$$ with $r_{n,i}=2(n-i)+1$. Notice that in this case the vertices are not equivalence classes of coding trees, but actual coding trees. Then, we define $\Lambda_v=v$ for every $v\in V(\mathcal{T})=\bigsqcup V_n(\mathcal{T})$ and set $$(\Omega_1,\Omega_2)\in E(\mathcal{T})$$ if and only if there exists $i,k\in\N$ so that $\Omega_1=B_{\Lambda_i}(v_i,r_{k,i})$ and $\Omega_2=B_{\Lambda_i}(v_i,r_{k+1,i})$.

\appendix
\section{Appendix: Coverings and the second systole}  \label{s.appendix}
\begin{center}
\author{S{\'e}bastien Alvarez,  Joaqu\'in Brum,  Matilde Mart\'inez,  Rafael Potrie \and  Maxime Wolff}
\end{center}

\medskip 

In this appendix we prove the Theorem below. A proof was first suggested by Henry Wilton on MathOverflow, as an answer to a question posed by the authors in that site (see \cite{HW}), while the proof presented here grew out of conversations between the authors.

\begin{theoremc}\label{t.appendix}

Let $\Sigma$ be a closed hyperbolic surface, and let $\alpha\subset\Sigma$
  be a simple closed geodesic. Then, for all $K>0$, there exists a finite 
  covering $\pi:\hat \Sigma\to\Sigma$ such that
  \begin{itemize}
  \item $\hat \Sigma$ contains a non-separating simple closed geodesic such that $\pi(\hat \alpha)=\alpha$ and $\pi$ restricts to a homeomorphism on $\hat \alpha$;
  \item every simple closed geodesic which is not $\hat \alpha$ has length larger than $K$.
  \end{itemize}
\end{theoremc}

\paragraph{ {\bf Adapted metrics --}} Recall that, on a compact manifold, any two Riemannian metrics are
bi-Lipschitz equivalent. It should be observed that the statement of
Theorem C does not depend on the metric chosen on $\Sigma$,
up to rescaling $K$. Therefore, we may as well suppose that the metric
on $\Sigma$ is adapted to the problem, and we will choose it in the following
way. 

We will say that the metric on $\Sigma$ is {\em adapted} to $\alpha$ if the two following conditions are satisfied

\begin{enumerate}
\item $\alpha$ realizes the unique systole of $\Sigma$;
\item the width of the collar about $\alpha$ is larger than the length of $\alpha$. 
\end{enumerate}

\paragraph{ {\bf Length spectrum --}} If $\Sigma$ is a compact hyperbolic surface, we denote by 
\[ LS(\Sigma) = \left( \ell_1(\Sigma),\ell_2(\Sigma),\ldots,\right) \]
the (unmarked) length spectrum of $\Sigma$ with multiplicity, but with
the restriction that we do not take any higher power of any curve realizing
the systole of $\Sigma$.
In other words, we enumerate, up to making choices, the unoriented closed
geodesic curves $\gamma_n\subset\Sigma$, with $n\geq 1$, that either realize
the systole of $\Sigma$, or meet such a curve only transversally if at all,
sort them by increasing length and set $\ell_k(\Sigma)$ to be the length
of $\gamma_k$.
Recall that the length spectrum of every compact hyperbolic
surface is discrete, and $\ell_k(\Sigma)\to+\infty$ as $k$ goes
to $+\infty$.

\paragraph{ {\bf Increasing the second systole --}} We now reduce Theorem C
to the following statement.
\begin{theorem}\label{teo:bis}
  Let $\Sigma$ be a closed hyperbolic surface, with metric adapted to some
  nonseparating simple closed geodesic $\alpha$. Then there exists a
  finite covering $\pi\colon\Sigma'\to\Sigma$ such that
$\alpha$ has a unique $(1:1)$-lift to $\Sigma'$ and $\ell_2(\Sigma')>\ell_2(\Sigma)$.
\end{theorem}

\begin{proof}[Proof of Theorem C assuming
Theorem~\ref{teo:bis}]
  First, up to starting with a cover of degree two, we may assume without
  loss of generality that $\alpha$ is nonseparating, and then we may choose
  a hyperbolic metric adapted to $\alpha$.
  Now let
  \[ d_1<d_2<\cdots \]
  be the (unmarked) length spectrum of $\Sigma$, without multiplicity,
  and without restrictions (\textsl{i.e.}, this time we consider all
  geodesic curves).
  Let $\pi^{(1)}\colon\Sigma^{(1)}\to\Sigma$ be a covering of $\Sigma$ as in
  Theorem~\ref{teo:bis}. It follows from the statement of this theorem
  that $\alpha$ admits a unique lift $\alpha^{(1)}$ such that
  $\pi^{(1)}$ is $(1:1)$ in restriction to $\alpha^{(1)}$, and we have
  $\ell_2(\Sigma^{(1)})\geq d_2$. In particular $\alpha^{(1)}$ is the
  unique systole of $\Sigma^{(1)}$, and it follows that the metric on
  $\Sigma^{(1)}$ is adapted to $\alpha^{(1)}$. Hence, we may apply
  Theorem~\ref{teo:bis} to $(\Sigma^{(1)},\alpha^{(1)})$, getting
  a covering $\pi_{(2)}\colon\Sigma^{(2)}\to\Sigma^{(1)}$, and so on.
  
  This yields a sequence of finite coverings $\Sigma^{(k)}\to\Sigma$.
  By construction, for all $k$, $\Sigma^{(k)}$ has a closed geodesic
  $\alpha^{(k)}$ mapping homeomorphically to $\alpha$, and the sequence
  of second systoles of $\Sigma^{(k)}$ is strictly increasing.
  
  Now the  geodesics of $\Sigma^{(k)}$ realizing this second systole project to
  geodesics of $\Sigma$; it follows that $\ell_2(\Sigma^{(k)})\geq d_k$ for
  all $k$. Since the length spectrum of $\Sigma$ is discrete, this sequence
  of coverings provides, for $k$ large enough, a covering satisfying the
  conclusion of Theorem C.
\end{proof}

\paragraph{ {\bf Product of coverings --}} Before carrying out the proof, let us recall a construction of the smallest common covering associated to a finite family of coverings.

\begin{definition}
\label{d:product_cover} Consider a finite family of finite covering maps $\pi^{(i)}:\Sigma^{(i)}\to\Sigma$, for $i=1,...,n$. The \emph{product} of these coverings is the map $\widehat{\pi}:\widehat{\Sigma}\to\Sigma$ where
\[ \widehat{\Sigma}=\left\lbrace (x,y_1,\ldots,y_n)\in
  \Sigma\times\Sigma^{(1)}\times\cdots\times\Sigma^{(n)}:\, \forall j,
  \pi^{(j)}(y_j)=x \right\rbrace \]and $\widehat{\pi}(x,y_1,\cdots,y_n)=x$.
\end{definition}

Notice that by the homotopy lifting property this cover is connected (though this will not be used in the proof). 

This notion allows us to reduce the proof of Theorem \ref{teo:bis} to  that of the following technical result which can also be seen as a straightening of residual finiteness.

\begin{lemma}\label{lem:2curvas}
  Let $\Sigma$ be a closed hyperbolic surface with metric adapted to
  some nonseparating simple closed curve $\alpha$. Let $\beta$ be a geodesic
  of $\Sigma$ realizing the length $\ell_2(\Sigma)$, and intersecting
  $\alpha$ only transversally. Let $N>0$ be an integer. 
  
  Then there exists  a finite covering $\pi_\beta\colon\Sigma'\to\Sigma$, such that
  $\beta$ admits no $(1:1)$ lifts, and such that $\alpha$
  admits a unique $(1:1)$ lift $\hat \alpha$, and such that
  all other lifts of $\alpha$ are at least $(N:1)$.
\end{lemma}

\begin{proof}[Proof of Theorem~\ref{teo:bis} assuming Lemma~\ref{lem:2curvas}]
  Let $\gamma_1,\ldots,\gamma_n$ be all the geodesics of $\Sigma$ involved
  in the definition of the length $\ell_2(\Sigma)$, and let $N\geq 1$ be
  such that $N\ell_1(\Sigma)>\ell_2(\Sigma)$.
  For each $j\in\{1,\dots,n\}$ let
  $\pi_{\gamma_j}\colon\Sigma^{(j)}\to\Sigma$ be a finite covering
  as provided by Lemma~\ref{lem:2curvas}. We will prove that the
  product covering of all these coverings satisfies the conclusion of
  Theorem~\ref{teo:bis}. Thus, let us consider $\hat \pi:\widehat{\Sigma}\to\Sigma$, the product of all     
   these coverings. Recall that a point of $\widehat{\Sigma}$ is denoted by $(x,y_1,\cdots y_n)$, $y_j\in\Sigma^{(j)}$ and that $\hat \pi$ is the projection on the first coordinate.
  
  By construction, $\alpha$ has a unique lift to $\widehat{\Sigma}$. It consists
  of points $(x,y_1,\ldots,y_n)$ such that $x$ lies in $\alpha$ and
  such that $y_j$ lies in the unique $(1:1)$ lift of $\alpha$ to
  $\Sigma^{(j)}$, for all $j$. It follows that $\hat \alpha$ is the unique
  systole of $\widehat{\Sigma}$, and  $\ell_2(\widehat{\Sigma})>\ell_2(\Sigma)$.
  Thus, let us consider a geodesic $\gamma$ of $\widehat{\Sigma}$ that may intersect
  $\hat \alpha$ only transversally (or equivalently, which is not a power of $\hat \alpha$), we have to prove that its length is
  $>\ell_2(\Sigma)$. 

  If $\gamma$ projects to $\alpha$ in $\Sigma$, then, as $\gamma$ intersects
  $\hat \alpha$ only transversally, there exists $j\in\{1,\ldots,n\}$ such that
  the image of $\gamma$ in $\Sigma^{(j)}$ is not the $(1:1)$ lift of
  $\alpha$. Hence, the length of $\gamma$ is at least $N\ell_1(\Sigma)$,
  which is (strictly) larger than $\ell_2(\Sigma)$.

  Otherwise, $\gamma$ projects to a curve $\hat \pi(\gamma)$ which may intersect
  $\alpha$ only transversally. By definition of $\ell_2(\Sigma)$, it follows
  that the length of $\hat \pi(\gamma)$ is at least $\ell_2(\Sigma)$, and
  equal to $\ell_2(\Sigma)$ only if $\hat \pi(\gamma)$ is one of the
  $\gamma_j$, $j\in\{1,\ldots,n\}$. But $\gamma_j$ does not have any
  $(1:1)$ lifts to $\Sigma^{(j)}$. It follows that $\gamma$ is not a
  $(1:1)$ lift of $\gamma_j$, hence the length of $\gamma$ is strictly
  larger than $\ell_2(\Sigma)$ in either case, and Theorem~\ref{teo:bis}
  is proven.
\end{proof}

\paragraph{ {\bf Curves realizing the second systole --}} We now use the assumption made on the metric in order to study curves that realize the second systole.

\begin{lemma}\label{lem:BetaCool}
  Let $\Sigma$ be a closed hyperbolic surface with metric adapted to
  a closed geodesic $\alpha$. Let $\beta$ be a curve intersecting
  $\alpha$ only transversally, and realizing the length $\ell_2(\Sigma)$.
  Then $\beta$ is simple, and its geometric intersection number with
  $\alpha$ is at most one.
\end{lemma}

\begin{proof}
  First, let us prove that $\beta$ is simple, by contradiction.
  Recall (see~\cite[Theorem~4.2.4]{Bu}) that if a curve realizes the
  minimum of the length among all primitive non-simple curves, then
  it is a figure eight. Obviously, the curve $\beta$ is primitive
  by assumption, hence this theorem applies: $\beta$ has a unique
  self-intersection point. Let us choose this intersection point as
  a base point and write $\beta=\beta_1 \beta_2$ as the concatenation
  of two closed loops, both geodesic except at the base point.
  Also, $\beta_3=\beta_1 \beta_2^{-1}$ is a non-trivial loop, geodesic
  except at the base point.
  
  By minimality of $\ell_2(\Sigma)$ and uniqueness of the curve realizing the systole on $\Sigma$,
  it follows that $\beta_1$, $\beta_2$ and $\beta_3$ are all freely homotopic
  to $\alpha$, hence $\beta_2$ and $\beta_3$ are conjugate to
  $\beta_1^{\pm 1}$. In the abelianization of $\pi_1(\Sigma)$, this gives
  a contradiction, as $\alpha$ was supposed to be non-separating, hence
  nontrivial in homology.
  
  Now, again by contradiction, suppose that $\alpha$ and $\beta$
  intersect at least twice. Then we may decompose $\beta$ as the
  concatenation of two geodesic segments, $\beta_1$ and $\beta_2$,
  with common endpoints on $\alpha$,
  and let $\alpha_1$ be a geodesic subpath of $\alpha$ of minimal
  length and joining these two endpoints of $\beta_1$ and $\beta_2$.
  As the collar around $\alpha$ is greater than its
  length, $\alpha_1$ is shorter than $\beta_2$. Hence the curve formed
  by $\alpha_1$ and $\beta_1$ is shorter than $\beta$, and it is essential
  and non-freely homotopic to $\alpha$
  (for otherwise $\alpha$ and $\beta$ would form a bigon). This contradicts
  the minimality assumption on~$\beta$.
\end{proof}

\paragraph{ {\bf Conclusion --}} Now we finish the proof of the theorem.
\begin{proof}[Proof of Lemma~\ref{lem:2curvas}]
  Thanks to Lemma~\ref{lem:BetaCool}, we are left with four possibilities:
  \begin{enumerate}
  \item $i(\alpha,\beta)=1$;
  \item $i(\alpha,\beta)=0$ and $(\alpha,\beta)$ is free in the homology
    of $\Sigma$;
  \item $i(\alpha,\beta)=0$ and $[\beta]=0$ in homology;
  \item $i(\alpha,\beta)=0$ and $[\alpha]+[\beta]=0$ in homology mod 2;
  \end{enumerate}
  these four cases are illustrated in Figure~\ref{fig:4casos}.
  
\begin{figure}[h!]
\centering
\includegraphics[scale=0.8]{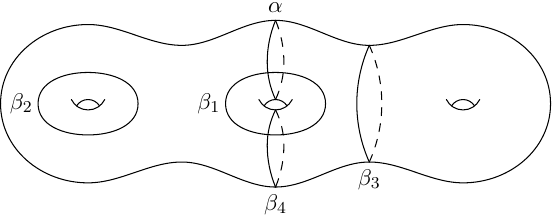}
\caption{The curve $\alpha$ and the four possible cases for $\beta$.}\label{fig:4casos}
\end{figure}

 Cases (1) and (2) are the easiest to deal with: in both these cases,
  we can find two disjoint, simple geodesics $\delta_1$, $\delta_2$
  with geometric intersection numbers
  $i(\alpha,\delta_2)=i(\beta,\delta_1)=0$ and
  $i(\alpha,\delta_1)=i(\beta,\delta_2)=1$.
  We can then cut $\Sigma$ along $\delta_1$ and $\delta_2$; this gives
  a surface $\Sigma_1$ with four boundary components. Finally we can
  glue $N+1$ pieces of $\Sigma_1$, along a graph as suggested
  in Figure~\ref{fig:2PrimCasos}, thus obtaining a surface $\widehat{\Sigma}$
  covering $\Sigma$ in a way that satisfies the lemma.
  Let us be a little more precise here. We may choose a
  co-orientation for the curves $\delta_1$ and $\delta_2$. This
  gives an orientation for the edges of the (figure eight) graph $\Gamma$
  dual to the cutting system $(\delta_1,\delta_2)$.
  
\begin{figure}[h!]
\centering
\includegraphics[scale=0.9]{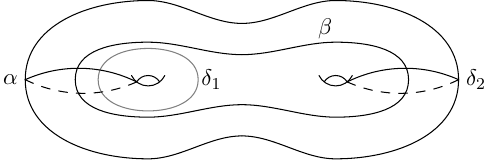}\hfill
\includegraphics[scale=1]{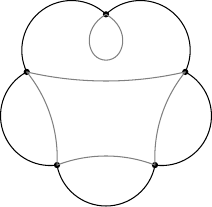}
\caption{Left: cutting $\Sigma$. Right: gluing $\Sigma'$ with pieces of $\Sigma_1$.}\label{fig:2PrimCasos}
\end{figure}
  
  Denote by $\langle a,b\rangle$ the fundamental group of $\Gamma$,
  where $a$ is the oriented edge dual to $\delta_1$ and $b$ dual to
  $\delta_2$.
  Any covering of $\Gamma$ gives rise to a covering of $\Sigma$, either by
  pulling back a pinching map $\Sigma\to\Gamma$, or equivalently,
  by thinking of the
  covering of $\Gamma$ as a set of instructions for gluing as many copies
  of $\Sigma_1$ as the vertices of the covering graph.
  The covering of $\Gamma$ suggested in Figure~\ref{fig:2PrimCasos}
  is associated to a morphism
  $\sigma\colon\langle a,b\rangle\to\mathfrak{S}_{N+1}$ (the symmetric group on $N+1$ elements)
  where $\sigma(a)$ has one fixed point and one cycle of length $N$,
  and $\sigma(b)$ has a cycle of length $N+1$.
  Now, up to choosing an orientation on them, the closed
  curves $\alpha,\beta$ yield two loops in this figure eight oriented graph:
  $\alpha$ yields the path $a$, hence $\alpha$ has one $(1:1)$ lift and
  one $(N:1)$ lift, while $\beta$ yields the path $b$, hence it has
  one $(N+1:1)$ lift, and this covering satisfies the conclusion
  of Lemma~\ref{lem:2curvas}.
  
  Cases (3) and (4) are similar, except that the curve $\beta$ cannot
  be mapped to a single generator $b$ in $\Gamma$, but to a  slightly
  longer word. In case (3), we can find two disjoint simple curves
  $\delta_1,\delta_2$ such that $i(\alpha,\delta_1)=1$, $i(\alpha,\delta_2)=0$,
  and $i(\beta,\delta_1)=i(\beta,\delta_2)=2$ as in
  Figure~\ref{fig:2UltCasos}. 

\begin{figure}[h!]
\centering
\includegraphics[scale=1]{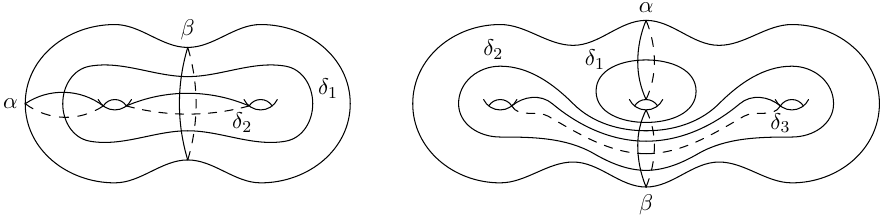}
\caption{The curve $\alpha$ and the four possible cases for $\beta$.}\label{fig:2UltCasos}
\end{figure}
  
  For some co-orientations of $\delta_1$
  and $\delta_2$, and some orientations on $\alpha$ and $\beta$,
  the loop $\alpha$ gives the loop $a$ in the graph $\Gamma$ as above,
  while the loop $\beta$ yields the word $aba^{-1}b^{-1}$. Thus as before,
  finding a cover satisfying the conclusion of Lemma~\ref{lem:2curvas}
  amounts to choosing two permutations $\sigma(a)$ and $\sigma(b)$,
  such that $\sigma(a)$ has one fixed point and one cycle of length $N$,
  say, $(2~3\cdots~N+1)$, and such that the commutator
  $\sigma(aba^{-1}b^{-1})$ has no fixed point: it suffices to choose
 $\sigma(b)$ so that $\sigma(ba^{-1}b^{-1})=(1~2\cdots N)$ (recall that all cycles of length $N$ are conjugated since for every permutation $B\in\mathfrak{S}_{N+1}$ and every cycle $A=(a_1~a_2\cdots~a_N)$ we have $BAB^{-1}=(B(a_1)~B(a_2)\cdots~B(a_N))$).

  Finally, in case (4), which may happen only if the genus of $\Sigma$
  is at least three, it is best to cut $\Sigma$ along three curves
  $\delta_1$, $\delta_2$ and $\delta_3$, as pictured in
  Figure~\ref{fig:2UltCasos}. This time $\alpha$ yields the word
  $a$, while $\beta$ yields the word $abcb^{-1}c^{-1}$ in the fundamental
  group of the graph $\Gamma$, which is this time a bouquet of three circles.
  We pick again $\sigma(a)$ to be the
  cycle $(2~3\cdots N+1)$, as before. As long as $N\geq 5$,
  the cycle $(1~3~5)$ is a commutator in $\mathfrak{S}_{N+1}$. For example we can write $(1~3~5)=BCB^{-1}C^{-1}$ where $B=(2~4~6)$ and $C=(1~2)(3~4)(5~6)$.
   Hence we may pick $\sigma(b)=B$ and $\sigma(c)=C$ so that
  $\sigma(bcb^{-1}c^{-1})=(1~3~5)$, and then
  $\sigma(abcb^{-1}c^{-1})$ has no fixed point: this yields a covering of
  $\Sigma$ satisfying the conclusion of Lemma~\ref{lem:2curvas}, in either
  case.  
  
\end{proof}

\end{document}